\documentclass[11pt]{amsart}
\usepackage{amsmath}
\usepackage{amsmath}

\usepackage{latexsym}
\usepackage[matrix,arrow]{xy}
\usepackage{slashed}
\usepackage{parskip}
\usepackage{indentfirst}
\usepackage{tikz-cd} 
\usepackage{amssymb}
\usepackage{color}
\usepackage{marginnote}
\usepackage{extarrows}
\usepackage{hyperref}
\usepackage{backref}
\usepackage{comment}
\excludecomment{removed}
\usepackage{graphicx} 
\newtheorem{theorem}{Theorem}[section]
\newtheorem{lemma}[theorem]{Lemma}

\newtheorem{remark}[theorem]{Remark}

\newtheorem{proposition}[theorem]{Proposition}

\numberwithin{equation}{section}

\newcommand{\C}{\mathbb{C}}
\newcommand{\Z}{\mathbb{Z}}
\newcommand{\R}{\mathbb{R}}

\newcommand{\af}{\mathfrak{a}}

\newcommand{\p}{\mathfrak{p}}
\newcommand{\kf}{\mathfrak{k}}
\newcommand{\e}{\epsilon}

\title
[Invariant
differential operators on  Hermitian symmetric spaces]{
Commutativity of
invariant
differential operators on vector bundles on Hermitian symmetric spaces}
\author{ Robin
van Haastrecht, Genkai Zhang and Yufeng Zhao}
\date{\today}

\address[Robin]{Department of Mathematics, Chalmers University of Technology and the University of Gothenburg, Gothenburg, Sweden}
\email{robinva@chalmers.se}

\address[Zhang]{Department of Mathematics, Chalmers University of Technology and the University of Gothenburg, Gothenburg, Sweden}
\email{genkai@chalmers.se}

\address[Zhao]{Department of Mathematics, Peking University,
  Beijing, China}
\email{Zhaoyufeng@math.pku.edu.cn}

\thanks{The research by
Genkai Zhang was partially
supported
by 
Research by  the Swedish Research Council (Vetenskapsr\aa{}det),
Grant 2022-02861, research
by 
Yufeng Zhao was supported by Chinese Ministry of Education Training Program, Grant No. 20222010.}

\begin{document}

\begin{abstract}
Let $G/K$ be
a Hermitian symmetric space and $V_\tau$ 
 an irreducible representation 
of $K$.
We study
the ring 
$\mathcal D^G(G/K, V_\tau)$
of  $G$-invariant differential operators
    on sections of vector bundles $G\times_{(K, \tau)} V_\tau$
    over  $G/K$
    defined by a
    finite-dimensional
    representation $(V_\tau, \tau)$ of $K$.
    We  classify irreducible representations 
    $(V_\tau, \tau)$ such that
    $\mathcal D^G(G/K, V_\tau)$  is commutative. 
       We 
 construct
 eigenfunctions
     for the differential operators and study the invariance property
    of the eigenvalues 
    under the Weyl group for the restricted real root system of $G$.
    \end{abstract}

\maketitle

\tableofcontents

\section{Introduction}

In the present paper we shall study invariant differential operators on homogeneous vector bundles over Hermitian symmetric domains, their eigenfunctions and  eigenvalues. 
Let $G$ be a
non-compact semisimple 
Hermitian Lie group and $K$ its maximal compact subgroup. Let
$V_\tau$
be a finite-dimensional irreducible 
representation of  $K$ and $C^\infty(G/K,V_{\tau})$ 
be the space of smooth sections 
of the homogeneous vector 
bundle 
$G\times_{(K, \tau)} V_\tau$
defined by $V_\tau$.
The algebraic properties
of 
the ring
$\mathcal D^G(G/K,V_{\tau})$ 
of invariant
differential operators on
$C^\infty(G/K,V_{\tau})$ is of fundamental interest, both in analysis on symmetric spaces 
 and in the study of universal
 enveloping algebras \cite{deitmar, helg, ricc, Sh}.
When $(K, \tau)$
is a character, i.e., a one-dimensional
representation, the ring is commutative, and 
Shimura has constructed
a linear basis 
$\{\mathcal L_{\mu}\}$
consisting of formally positive 
operators
by using the Schmid decomposition
of the symmetric tensor
algebra
$S(\mathfrak p^+)$
under $K$, where 
$\mathfrak p^+$
is the holomorphic tangent space of $G/K$
at $o=K\in G/K$.
For general 
$V_\tau$ it is proved by Deitmar
\cite{deitmar} 
that 
$\mathcal D^G(G/K, V_\tau)$
is commutative
if and only if 
the restriction  $\tau|_M$
to $M$ is multiplicity-free, where $M$
is the centralizer in $K$ of the real Cartan group of $G$. 
In the present paper we shall
classify all such representations
$(K, \tau)$. 
We construct 
 eigenfunctions of invariant
differential operators
using the Szegö transform and
we study the invariance
properties of the eigenvalues
of invariant differential
operators. This question can be posed for any symmetric space, we study the Hermitian case because there are canonical constructions of invariant differential operators related to the Hua-Kostant-Schmid decomposition \cite{Sh} of invariant differential operator. It is an interesting problem to further study the eigenvalue problem for those operators;
see  \cite{SZ}
for the case of one-dimensional representations $V_\tau$ of $K$. 

\subsection{Main
results and methods}

The main result of this paper
is the following theorem; see
below for the exact notation
and definitions.

\begin{theorem} (Theorem \ref{KresThm}, Proposition  \ref{finiterepprop}, Theorem \ref{egvaluethm})
(1) The representation 
$\tau|_M$ is
multiplicity-free if and only if 
$(K, \tau)$ is in the list in Theorem \ref{KresThm}.

(2) Let $(K, \tau)$ be as
in (1) acting on $V_\tau$. For any finite-dimensional irreducible representation
$(W_\lambda, G^{\C})$ there is up to scalar at most one $G$-equivariant map into $C^\infty(G/K, V_\tau)$.

(3)
Let $(K, \tau)
$ be as in (1), $\sigma$
an irreducible representation in  $\tau|_M$
and 
$(W_\lambda, G^{\mathbb C})$ 
as
in (2). Let $D$ 
be an invariant
differential operator
on $C^\infty(G/K, V_\tau)$.
Write   the eigenvalues 
of  $D$ on $W_\lambda$
as $\omega(D)(\sigma_{\lambda},-\lambda|_{\mathfrak{a}}
)$
which can be extended uniquely to a polynomial of $
\lambda|_{\mathfrak{a}}
\in 
(\mathfrak a^{\mathbb C})'
$.
Then the eigenvalue
$\omega(D)$ has following
invariance 
property,
$$\omega(D)(\sigma,\kappa + \rho_{\mathfrak{a}}) = \omega(D)(s \cdot \sigma,s \cdot \kappa + \rho_{\mathfrak{a}}), \ \forall \ \kappa \in 
(\mathfrak a^{\mathbb C})', s\in W.$$
\end{theorem}

The proof of the 
classification
is through a case by case
computation
and is technically involved.
We shall repeatedly use the classical branching rules in
\cite{good},
the classifications
of weight-free
representations
in \cite{Howe}
and multiplicity-free results in
\cite{St}.
The proof 
for the uniqueness
of $F: W_\lambda \to 
C^\infty(G/K, V_\tau)
$  is somewhat
natural,
it is done through
the realization
$J: W_\lambda 
\to Ind_{MAN}^G(\sigma\otimes \alpha\otimes 1)$
of finite-dimensional
representations $W_\lambda$
in the induced representation
in $Ind_{MAN}^G(\sigma\otimes \alpha\otimes 1)$,
a factorization
of $F=SJ$
as a product of $J$
and the Poisson-Szeg\"o{}
transform, and  the multiplicity-free property of $V_\tau$
under $M$. The result
on the invariance
of eigenvalues
of  differential
operators is 
a consequence of a
general result of  Lepowsky \cite{L}
and the factorization.

\subsection{Related
results and questions}

When  $ \tau$ is a character Shimura \cite{Sh} 
has constructed
a system of $r$ generators
for the ring $\mathcal D^G(G/K, V_\tau)$ 
and proposed several
questions on determining
the domains of positivity
for the eigenvalues
of the generators
and of the 
whole linear basis
$\{\mathcal L_{\mu}\}$
of formally positive
operators. A partial
answer to these
questions has been obtained
in \cite{SZ, Zhang}. Here invariance of eigenvalues under the Weyl group was used. 
We may also pose similar
questions as above for $(K, \tau)$
in the list of our classification.  
This is related to the antihomomorphism $\Omega:U(\mathfrak{g})^{K}\rightarrow  U( \mathfrak a^{\mathbb C}) \otimes U(\mathfrak k^{\mathbb C})^{M}$ 
in Section \ref{sect-4}; see \cite{ACT1, ACT2}
for the case
of real rank 
one groups.

The computation
of eigenvalues
of invariant 
differential
operators is closely
related to the 
problem of  characterizing
finite-dimensional representations
of $G$ containing
a fixed $K$-type $\tau$.
When $\tau$
is a character there
is the Cartan-Helgason
theorem and its generalization
by Schlichtkrull \cite{schlichtkrull}. On
the analytic side there is the explicit Plancherel formula for
the $L^2(G/K, V_\tau)$
by Shimeno \cite{Shimeno}.
It might be possible
to find a characterization
of the finite-dimensional
$G$-representations
containing $(\tau, K)$
and to find
the discrete
series $(\pi, G)$
contained in 
$L^2(G/K, V_\tau)$,
in this case 
they appear with multiplicity
at most one (see e.g. \cite[Prop. 2.2, (40)]{Campo},
\cite{ricc} or \cite[Vol. II, Proposition 6.1.1.6]{war}).

\subsection{List of notation and
symbols}

\begin{enumerate}
\item $G/K$: Hermitian
symmetric space.

\item 
$
\mathfrak g=\mathfrak
k +\mathfrak p$,
$\mathfrak g^{\mathbb C}=
\mathfrak p^{-}
+\mathfrak
k^{\mathbb C}
+\mathfrak p^+
$: Cartan and
Harish-Chandra decompositions of the Lie algebra $\mathfrak g$ of $G$ and its
complexification $
\mathfrak g^{\mathbb C}$.

\item $\mathfrak t^{\mathbb C}
=\mathfrak
(t^-)^{\mathbb C} 
\oplus 
\mathfrak (t^+)^{\mathbb C} 
\subset 
\mathfrak
k^{\mathbb C}$:
Cartan
subalgebra of $\kf^{\mathbb C}$, and
$\gamma_r>\cdots >\gamma_1$ are the
Harish-Chandra strongly
ortogonal roots
with $\gamma_j|_{
\mathfrak t_{\mathbb C}^+}=0 $ and 
 $\gamma_r$ being
 the highest non-compact
 root.

\item 
 $\mathfrak a
 \subset\mathfrak p$:
 maximal abelian 
subspace (Cartan subspace) of $\mathfrak p$;
$c: \mathfrak a^{\mathbb C}
\to \mathfrak (t^{-})^{\mathbb C}
$ is then the Cayley transform and
$\alpha_j = \gamma_j\circ c$ is the
Cayley transform
of the Harish-Chandra
orthogonal roots $\gamma_j$
to $\mathfrak a^{\mathbb C}$. Furthermore,
$M\subset K, M'\subset K$, $W=W(\mathfrak g, \mathfrak a)$ are respectively the centralizer
and normalizer of $\mathfrak a$
in $K$, and the Weyl group. Also, $\mathfrak h^{\mathbb C}=
\mathfrak a^{\mathbb C}
+\mathfrak (t^{+})^{\mathbb C}$.

\item 
$(W_{\lambda}, G^{\C}),
(V_\tau, K),
(U_\sigma, M)$:
 Finite-dimensional
irreducible
representations
of the respective
groups, the highest
weight $\lambda$
being defined
on the Cartan subalgebra $\mathfrak{h}^{\C}$ of $\mathfrak{g}^{\C}$, and
highest weight of
$\tau$ being defined
on the Cartan subalgebra $\mathfrak{t}^{\C}$.

\item $G=NAK$, 
$g=n(g)e^{H(g)}k(g)$: the Iwasawa decomposition 
of $G$ with $n(g)\in N, H(g)\in \af, k(g)\in K$ and $g \in G$.

\item $P=MAN$:
minimal 
parabolic subgroup of $G$.

\item $I_{\sigma, \nu}=Ind_{P}^G(\sigma\otimes e^{\nu}\otimes 1)$
: normalized
induced representation
of $G$ from $P$,
the corresponding
representation 
of $\mathfrak g^{\mathbb C}$ on the subspace of $K$-finite vectors will
also be written as
$I_{\sigma, \nu}^0$.

\item $C^\infty(G/K,V_{\tau})$: 
space of
smooth sections
of the vector bundle
$G\times_K V_\tau$
over $G/K$
defined by
the representation $(K, \tau)$. We let $\mathcal D^G(G/K,V_{\tau})$ then be the ring 
of $G$-invariant
differential operators on 
$C^\infty(G/K,V_{\tau})$.

\end{enumerate}

\subsection*{Acknowledgments}
We would like to thank Pavle Pandzic for informing 
us that some of our results
on the classifications
of $(V_\tau, \tau)$
have also been obtained
earlier by his joint work with 
Soo-Teck Lee and for some further discussions.

\section{Preliminaries}

We present
some known
facts on Hermitian symmetric spaces $G/K$, induced
representations of $G$, finite-dimensional
representations
of $G$ as subrepresentations
of induced representations, and
commutativity of
invariant
differential
operators on vector
bundles over $G/K$.

\subsection{Hermitian symmetric spaces}

Let $(G, K)$ be an
irreducible Hermitian symmetric pair with $G$ being a connected real simple noncompact Lie group and $K$
its maximal compact subgroup.
Let $\mathfrak{g}$
be its Lie algebra. Assume $G\subseteq G^{\mathbb{C}}$, where $G^{\mathbb{C}}$ is a connected complex Lie group with Lie algebra $\mathfrak{g}^{\mathbb{C}}
=\mathfrak{g}
+ i\mathfrak{g}$, the complexified 
Lie algebra of 
$\mathfrak{g}$. Let 
$\mathfrak{g}=\mathfrak{k}+\mathfrak{p}$ be a Cartan decomposition
with $\mathfrak p$
being identified
with the
real tangent
space $T_o(G/K)$
at $o=K\in G/K$.

Let $\mathfrak t$ be a Cartan subalgebra of $\mathfrak{k}$, then $\mathfrak t$ is also a Cartan subalgerba of $\mathfrak{g}$. Let
$\triangle \subseteq i\mathfrak t'$ consist of the roots of $\mathfrak t$ in $\mathfrak{g}$, and let $\mathfrak{g}_{\gamma}\subseteq \mathfrak{g}^{\mathbb{C}} $, $\gamma\in \triangle$,
denote the
corresponding  root space. 
Let $\triangle_c$, respectively $\triangle_n$ be the set of compact, respectively noncompact roots, i.e. those roots $\gamma$ for which $\mathfrak{g}_{\gamma}\subseteq \mathfrak{t}_{\mathbb{C}}$ respectively 
$\mathfrak{g}_{\gamma}\subseteq \mathfrak{p}_{\mathbb{C}}$. Then $\triangle=\triangle_c\cup \triangle_n$.

We fix  $Z$
in the center $ \mathfrak{z}$ 
of $\mathfrak{k}$ 
such that $\gamma(Z)=\pm i$ for all $\gamma\in  \triangle_n$. 
Then $\mathfrak k
=\mathbb R Z 
\oplus
\mathfrak k_1$,
where $\mathfrak k_1=[\mathfrak k, \mathfrak k]$ is 
an ideal of $\mathfrak k$.
Let
$\mathfrak k^{\mathbb C}
=
\mathfrak p^-
+
\mathfrak k^{\mathbb C}+
\mathfrak p^+
$ be the corresponding
Harish-Chandra
decomposition, so
that $\mathfrak p^+$
is the space of
positive root
vectors
and is identified
as the 
space 
$T^{(1, 0)}_o(G/K)$
of holomorphic
tangent vectors
of $o=K\in G/K$.
Choose an ordering of $\triangle$ such that $$\triangle_{n}^{+}=\{\gamma\in \triangle \ | \ \gamma(Z)=i\}=\triangle_{n}^{+}=\triangle^{+}\cap\triangle_{n}.$$
Let 
$\rho $
be the half sum of positive roots
in $\triangle^+$.

For $\phi \in (\mathfrak{t}^{\mathbb{C}})'$, let $H_{\phi}\in \mathfrak{t}^{\mathbb{C}} $ be defined by $\phi(H)=(H_{\phi},H)$ for all $H \in \mathfrak{t}$, where $(\cdot,\cdot)$ denotes the Killing form. Let $\{\gamma_1,\cdots,\gamma_r\}\subseteq \triangle_n$
 be the Harish-Chandra strongly orthogonal roots
 such that $\gamma_r$
 is the highest root,  $r$ being the (real) rank of $G$. Let $$\mathfrak
t^{-}=\sum\limits_{j=1}^{r}\mathbb{R}iH_{\gamma_j}$$
and $$
\mathfrak t^{+}=\{H \in t \ |  \ \gamma_j(H)=0, j=1, \cdots, r\}$$
Then $
\mathfrak t=
\mathfrak t^{+}+
\mathfrak t^{-}$.

\subsection{
Cayley transform between Cartan subalgebras}

For each non-compact root
$\gamma$ we choose
a root vector $0 \neq X_{\gamma}\in \mathfrak{g}_{\gamma}$ subject to $\gamma([X_{\gamma}, X_{-\gamma}])=2$, $\overline{X_{\gamma}}
={X_{-\gamma}}$
where $X\to \overline X$ is
the conjugation with
respect to $\mathfrak g$.
Let 
$$\mathfrak{a}=\sum\limits_{j=1}^{r}\mathbb{R}(X_{\gamma_j}+X_{-\gamma_j})$$then $\mathfrak{a}$ is a maximal abelian subspace of $\mathfrak{p}$. 
Let $\triangle(\mathfrak g, \mathfrak a)^+$
be the set of positive
roots of $\mathfrak a$ in
$\mathfrak g$
and $\mathfrak n$ the
subspace of positive root vectors. Let $\mathfrak a \oplus \mathfrak m$
be the centralizer
of $\mathfrak a$ in $\mathfrak g$. 
Then 
$$
\mathfrak t^+\subset \mathfrak m^{\mathbb C}
$$
is a Cartan subalgebra of $\mathfrak m^{\mathbb C}$
and let $\rho_{\mathfrak m^{\mathbb C}, \mathfrak t^+}$
be the half sum of positive roots
of $\mathfrak t^+$
in $m^{\mathbb C}$,
the positivity
being the same as
in $\Delta^+$.
Let $M=Z_K(\mathfrak a)$ be the centralizer of $\mathfrak a$ in $K$, and $A, N$
the corresponding
connected
subgroups with Lie algebras $\mathfrak a$, $\mathfrak n$ respectively. Let
$W=W(\mathfrak g, \mathfrak a)
=M'/M$ be the corresponding
Weyl group. Here $M'
=N_K(\mathfrak a)$
is the normalizer of $\mathfrak a$ in $K$.

Let $c$ be the automorphism of $\mathfrak{g}^{\mathbb{C}}$
given by $$c=\mbox{Ad}
(\mbox{exp}\frac{\pi}{4}\sum\limits_{j=1}^{r}(X_{\gamma_j}-X_{-\gamma_j})).$$
Then $c$ maps
$\mathfrak{a}$ bijectively to $i\mathfrak t^{-}$,
$$
c(X_{\gamma_j} +X_{-\gamma_j})
= H_{\gamma_j}
$$
and 
is the identity
map on $\mathfrak t^{+}$;
see e.g. \cite[Proposition 10.6]{loos}.
We define another
Cartan subalgebra
$\mathfrak h^{\mathbb C}
=\mathfrak a^{\mathbb C}
\oplus
\mathfrak (t^+)^{\mathbb C}
$
of $\mathfrak g^{\mathbb C}$ and
the linear functionals
\begin{equation}
\alpha_j=
\gamma_j\cdot c|_{\mathfrak a},\,
\rho_{\mathfrak a}= \rho|_{c^{-1}(\mathfrak a)},\,
\rho_{\mathfrak t^{+}}= \rho|_{
    \mathfrak t^{+}}.
\end{equation}
In particular we have $\alpha_i(X_{\gamma_j}
+X_{-\gamma_j})=
\gamma_i(H_{\gamma_j}
)=2\delta_{ij}$, $1\le i, j\le r$
and $\rho_{\mathfrak a}$
is the half-sum of positive
real roots of $
{\mathfrak a}$ in $\mathfrak g$.

\subsection{Induced
representations
from the minimal parabolic
subgroup $MAN$
}

For an irreducible unitary representation $(\sigma, W_{\sigma})$ of $M$, and a linear form $\nu \in (\mathfrak{a}^{\mathbb C})'$ the induced representation
$\pi_{\sigma,\nu}
=Ind_{MAN}^G(\sigma\otimes e^{\nu}\otimes 1)
$ is defined on the space
(using the same notation)
$$I_{\sigma,\nu}:
=\{f:G\rightarrow W_{\sigma} \ | \ f(xman)=e^{-(\nu+\rho_{\mathfrak a})(H(a))}\sigma^{-1}(m)f(x), x \in G,$$
$$ man = m e^{H(a)}n\in MAN; f|_K\in L^2(K)\}.$$
The group $G$ acts on 
$I_{\sigma, \nu}$
by the regular representation
$$
\pi_{\sigma, \nu}(h) f(g)
=
f(h^{-1}g), f \in
I_{\sigma,\nu},
$$
and there is a $K$-invariant
inner product
$$(f_1, f_2)=\int_{K}(f_1(k),f_2(k))_{\sigma} dk, f_1, f_2\in I_{\sigma,\nu} $$where $dk$ is the normalized Haar measure on $K$ and $(\cdot,\cdot)_{\sigma}$ the $M$-invariant inner product on $W_{\sigma}$. 
When $\nu$ is 
unitary, i.e., $\nu\in i\mathfrak a'$
then this is a
unitary representation of $G$, 
also called the unitary principal series.
We denote
$I_{\sigma,\nu}^{0}$ the  space of $K$-finite vectors for $I_{\sigma,\nu}$. 
We recall \cite[Proposition 8.22, p. 225]{knapp} that
the infinitesimal
character of
$I_{\sigma, \nu}$
is $(\sigma +\rho_{\mathfrak m^{\mathbb C}, \mathfrak t^+})
\oplus  \nu,
$
on $\mathfrak t^+
\oplus \mathfrak a^{\mathbb C}$,
where $\sigma$ is identified with its highest weight,

As a $K$-module
$I_{\sigma,\nu}^{0}$ is a direct sum
$$I_{\sigma,\nu}^{0}\cong \oplus_{\tau \in \hat{K}}\mbox{Hom}_{M}(\tau|_{M}, \sigma)\otimes V_{\tau}.$$
The isomorphism 
can be  given 
explicitly 
as follows. For
$T\in \mbox{Hom}_{M}(\tau|_{M}, \sigma)$ and $v\in V_{\tau}$, we define an element 
$f_{\tau, T, v}\in I_{\sigma,\nu}^{0}$ by 
$$f_{\tau, T, v}(kan)=T(\tau(k)^{-1}v)e^{-(\nu+\rho_{\mathfrak a})(H(a))}$$This assignment $T\otimes v \rightarrow f_{\tau, T, v}$ 
induces the isomorphism.

\subsection{
Realization
of irreducible
finite-dimensional
representations
$(W_\lambda, 
G^{\C})$
in the induced
representation
$(
I_{\sigma, \nu}^0, 
\mathfrak g^{\C})$.
}

Let $(W_{\lambda}, G^{\C})$ be an irreducible finite-dimensional representation of $\mathfrak{g}^{\C}$ with highest weight $\lambda$ on $\mathfrak{h}^{\C}$. Let $W_{\lambda}^{\overline{N}}$ be the
subspace of $\overline{N}$-invariant vectors. It is
always non-zero and forms  an irreducible $M$-representation which we denote by $\sigma_{\lambda}$. Then there is an embedding
\begin{equation}
\label{Jlambda}
J_{\lambda}: W_{\lambda} \rightarrow I_{\sigma_{\lambda}, -\lambda|_{\mathfrak{a}} - \rho_{\mathfrak{a}}}^0,
\end{equation}
see e.~g. \cite[Lemma 8.5.7]{Wall} and \cite[Theorem 3.1]{Y}. Identifying $W_{\lambda}^{\overline{N}}$ with $(W_{\lambda}'^{N})'$, where $W_{\lambda}'^{N}$ are the $N$-fixed vectors in $W_{\lambda}'$, the map is given by
$$(J_{\lambda}(w)(g))(v) = v(\pi_{\lambda}(g^{-1})w)$$
for $v \in W'^{N}$.

\subsection{The algebra of invariant differential operators on homogeneous vector bundles}

Let $\tau$ be an 
irreducible finite-dimensional
representation of $K$ on a complex vector space $V_\tau$. Let
$C^{\infty}(G, \tau)$ 
be the space 
of $V_\tau$-valued
smooth functions on $G$.
Consider the vector
bundle $G\times_K V_\tau$ over  $G/K$ 
defined by  $\tau$, where $K$ acts on $G\times V_\tau$
by $(g, v)\mapsto (gk^{-1}, \tau(k)v)$. The space of $C^\infty$-sections of the 
bundle is realized
as the subspace
$C^{\infty}(G/K, {\tau})
\subset 
C^{\infty}(G, \tau)$ 
of $f\in 
C^{\infty}(G, \tau)$ such that
$$f(g k)=\tau(k)^{-1}f(g), \ \forall g \in G, k \in K.$$

Denote by $\mathcal{D}^G(G/K,V_{\tau})$ the space of $G$-invariant
differential operators of $C^\infty(G/K,V_{\tau})$. This will be the main object of our study.
We recall the following result 
\cite[Theorem 3]{deitmar}.

\begin{proposition}
\label{comm-mult-one}
   The  ring   $\mathcal{D}^G(G/K,V_{\tau})$
   is commutative if
   and only if 
   the restriction 
   $V_\tau|_M$
   is multiplicity-free.
\end{proposition}

We recall some basic known
facts about 
the relation between 
the universal algebra $U(\mathfrak{g}^{\C})$
and invariant differential operators.
The Lie algebra $\mathfrak g^{\C}$ acts on $C^{\infty}(G, V_{\tau})$ by extending
differentiation
of the  right regular action,
 \begin{equation}
  \label{nabla-Xequation}
 \nabla f(X) (g)=
 \nabla_X f (g):= 
 \frac{d}{dt}
 f(g
 \exp(tX))|_{t=0},  \ \forall \  X \in \mathfrak{g},
 \end{equation}
  to $X \in \mathfrak{g}^{\C}$. 
  Note
 that $X\to \nabla_X $ is a Lie algebra
 anti-homomorphism between $\mathfrak g^{\C}$ and the Lie algebra of vector fields 
 on $G$,
 $$
\nabla_{[X, Y]}= -[\nabla_X, 
\nabla_Y]. 
 $$
It  can
be further extended
to  $U(\mathfrak{g}^{\C})\otimes \mbox{End}V_{\tau}$ 
by
 $$
 \nabla_{\sum_i X_i\otimes T_i}f(g)
=\sum_i T_i(\nabla_{X_i}f(g)).$$
Restricting to the
$K$-invariant elements we have
\begin{flalign}
\label{diffopUnAlg}
 \nabla:  (U(\mathfrak{g}^{\C}) \otimes \mbox{End}V_{\tau})^{K} \rightarrow \mathcal{D}^G(G/K,V_{\tau}).
\end{flalign}
This map is surjective, see e. g. \cite{Ya} and \cite[Section 2]{Sh}.

\section{
Commutativity
of $\mathcal D^G(G/K,V_{\tau})$
and the 
corresponding
classification
of $(K, \tau)$}
\label{CommList}

In this section, we will 
prove the following theorem classifying all irreducible finite-dimensional  representations $\tau$ of $K$ that are multiplicity-free  when  restricted to  $M$, which by Proposition
\ref{comm-mult-one}
implies  the
commutativity of
the algebra $\mathcal D^G(G/K,V_{\tau})$. This will be done through case-by-case computations. 
Given a symmetric space $D=G/K$ the symmetric 
pairs $(G, K)$
of groups are determined
only up to finite
covers, 
so our problem on the classification
is for  fixed choices
of $K$
and their representations,
and we shall
consider all
representations $\tau$
of the Lie algebras $\mathfrak k$
and fixed groups $K$
with Lie algebra
$\mathfrak k$; in particular the 
spin representations
of $Spin(n)$
are considered for Types BD and $E_6$.
Below by character
of $K$ we mean a one-dimensional representation.

\begin{theorem}
\label{KresThm}
Let $(G, K)$ be an irreducible Hermitian symmetric pair
and
 $M = C_K(\mathfrak{a})$
 be the centralizer of the
 Cartan subalgebra 
 $\mathfrak{a}\subset \mathfrak p$ in $K$. 
The complete list of $(K, \tau)$ with $\tau|_M$ being multiplicity-free is given as follows:

\begin{enumerate}
    \item Type A, 
    $G = SU(r+b,r), 
    K = S(U(r+b) \times U(r))$; irreducible representations of $U(n)$ with highest
    weights $\mu$
    will be denoted by 
    $\pi^{\mu}_n$ as  in \cite[Theorem 5.5.22]{good}. 
\begin{enumerate}
    \item $r=1$. $\tau$ is arbitrary.
\item  $r=2$. $\tau= \pi^{\mu}_{2+b} \otimes \pi^{\nu}_2$ where $\pi_{2+b}^{\mu}$ is
a character and $\pi_2^{\nu}$ is any representations, or $\pi_2^{\nu}$ is 
a character and
$$
\mu = (\overbrace{l_1+l_2,\cdots, l_1+l_2}^{{j }},\overbrace{l_2,\cdots,l_2}^{{b+2-j }}).
$$

\item  $r \geq 3$. $\tau= \pi^{\mu}_{2+b} \otimes \pi^{\nu}_2$ where $\pi_{r+b}^{\mu}$
or $\pi_r^{\nu}$ is a character and the other must be the symmetric
tensor powers
$S^m(\mathbb C^{p})$ 
or  an exterior power
$
\wedge^j \mathbb C^p$ when restricted to $SL(p)$
($p=r+b$ or $p=r$
accordingly).
\end{enumerate}

\item Type C,  $G=Sp(n,\R)$, $K = U(n)$. $\tau$ is a character 
or its tensor product
with  the exterior power representations
$
\wedge^j \mathbb C^n$. The same holds for the double cover $(G,K) = (Mp(n,\R),\widehat{U(n)})$.

\item Type D,  $G = SO^*(2n), K = U(n)$. $\tau$ 
is a character
or its tensor product
with  the symmetric tensor
powers
$S^m(\mathbb C^{p})$  or their duals.

\item Type BD,  $G=SO_0(n,2)$, 
 $K = SO(n) \times SO(2)$, $n>4$;
irreducible 
representations
of $SO(n)$
will be denoted
by 
$\pi_{n}^{\mu}$
as in \cite[Theorem 19.22]{FulHar}.
 $\tau=\pi_{n}^{0}\otimes \pi_{2}^\nu$ 
 if  $n$ is odd, i.e., $\tau$ is a character; for $n$ even $\tau=
 \pi_{n}^{\mu}$ 
 or  its 
 tensor product
 with a character
$\tau=\pi_{n}^{\mu}
\otimes \pi_{2}^\nu$,  
 where $\mu=(a,\dots,a,\pm a)$ for some $a\ge 0$.
If $(G,K) = (Spin_0(n,2), 
Spin(n) \times SO(2)/\Z_2)$, where $Spin_0(n,2)$ is the double cover of $SO_0(n,2)$
then up to a character $\tau$ is the spin representation $(\frac{1}{2}, \dots, \frac{1}{2})$ when $n$ is odd, and
$\tau= \pi_{n}^{\mu}$, $\mu=(a,\dots,a,\pm a)$ for some integer $a\ge 0$ a half-integer
when $n$ is even.

\item Type $E_6$, $G=E_{6(-14)}, K = Spin(10) \times U(1) / \Gamma$, where $\Gamma$ is a finite subgroup.
$\tau=\pi_{\omega_1} \otimes \chi_n$ for  $n \equiv 1 \mod 4$, or $\pi_{\omega_2} \otimes \chi_n$ for $n \equiv 3 \mod 4$. Here $\omega_1$ and $\omega_2$ are the positive and negative spin representations of $Spin(10)$ and $U(1) = \{e^{it \frac{Z}{2}} \ | \ t \in \R \}$.

\item Type $E_7$,  $G=E_{7(-25)}, K = E_6  U(1)$.
$\tau$ is a character.

\end{enumerate} 
\end{theorem}

\subsection{$(G,K) = (SU(r+b,r), S(U(r+b) \times U(r))$}
We first prove Theorem \ref{KresThm} (1). $K$ is realized as diagonal block matrices 
$\mbox{diag} (A, B)$
which will be written
as $(A, B)$. Its Lie algebra
is
$$\mathfrak{k}=\{(A, B) \ | \ A \in \mathfrak{u}(r+b), B\in \mathfrak{u}(r), \mbox{tr}(A)+\mbox{tr}(B)=0\}.$$
The Cartan decomposition
for $\mathfrak g$ is
$\mathfrak g=\mathfrak k+
\mathfrak p$ 
with $\mathfrak{p}
=M_{r+b, r}(\mathbb{C})$ as real spaces and with $K$
acting on $
\mathfrak{p}$
by $(A, B): X\mapsto AXB^{-1}$.
We choose
a maximal abelian subspace of $\mathfrak{p}$ as 
$$\mathfrak{a}=\sum\limits_{i=1}^{r}\mathbb{R}E_{i,r}.$$
Then the group $M = C_{K}(\mathfrak{a})$ is
$$M=\{(d_1, \cdots, d_r; C; d_1, \cdots, d_r)\ | C \in U(b), d_i\in U(1), (d_1, \cdots, d_r)^2\mbox{det}(C)=1\}$$
as $3\times 3$-block diagonal 
matrices 
with the convention that $C$ will not appear if $b=0$.
Its complexification is
$$M^{\mathbb C}=\{(d_1, \cdots, d_r; C; d_1, \cdots, d_r)\ | C \in GL(b, \mathbb C), d_i\in \mathbb C, (d_1, \cdots, d_r)^2\mbox{det}(C)=1\}.$$

\begin{removed}
Note that  the Lie algebra  $\kf $ is a direct
sum of three ideals,
$$ \kf 
= \{  (X_1, X_2)
 \ | \ X_1^* = -X_1, X_2^* = -X_2 \} = \R Z \oplus \mathfrak{su}(r+b) \oplus \mathfrak{su}(r),$$
where $Z = \mbox{diag} 
(i r I_{r+b}, - i (r+b) I_r)$ spans the center of  $\kf$. Hence any element $(A, B)\in \mathfrak k$ 
can be decomposed as a sum of three
elements in the
ideals,
$$(A, B) =
\frac{1}{r+b} \mbox{tr}(A)
\, Z
+(A-\frac{1}{r+b}\mbox{tr}(A)I_{r+b}, 0)+
( 0, B-\frac{1}{r}\mbox{tr}(B)I_{r}),
$$
\end{removed}

Irreducible representations of $U(n)$ are in one-to-one correspondence with $GL(n, \mathbb C)$ and we shall
only consider $GL(n, \mathbb C)$ representations
and the corresponding branching problems
for the complexified groups. To save notation we write $GL(n)=GL(n, \mathbb C)$.
We shall first classify irreducible finite-dimensional representations $\tau$ of  
  $ GL(r+b)\times GL(r)$ 
with 
multiplicity-free restriction
$\tau|_{ GL(1)^{r}    \times GL(b)\times GL(1)^{r}}$.

\begin{removed}Note that we add a character to $K^{\mathbb{C}} = S(GL(r+b) \times GL(r))$; this will not influence the calculations but will make it easier to use known formulas. Here $K$ and $M$ are connected compact Lie
groups and by the unitary trick \cite[Proposition 7.15]{knappliegrps}, we need only classify the complex irreducible representation of $\kf_{\mathbb{C}} $ that are multiplicity-free when restricted to its Lie subalgebra
\begin{flalign*}
& \mathfrak{m}^{\mathbb{C}}=
\{(d_1, \cdots, d_r; C;
d_1, \cdots, d_r)  |
\\ & \ \ \ \ \ \ \ \ \ \ 
C \in \mathfrak{gl}(b,\mathbb{C}), d_i\in \mathbb{C},  
\mbox{tr}(C) +
2(d_1+\cdots + d_r)=0\}.
\end{flalign*}
\end{removed}

Any irreducible 
representations $\pi$
of $GL(n)$ 
is parametrized
by its highest weight
$\mu=(\mu_1\ge \cdots\ge \mu_n)$
as an $n$-tuple of
integers,
$\pi=\pi_{n}^{\mu}$,  and up to a character we can assume they are nonnegative, 
i.e. $\mu$ is Young diagram
\begin{equation}
\label{highest-weight-gln}
\mu = (\mu_1 \geq \dots \geq \mu_n\ge 0).
\end{equation} 

We recall first that if 
an irreducible representation
$\pi_{n}^{\mu}$
of 
$GL(n)$
is multiplicity-free when restricted to 
its Cartan subgroup, then  up to a character
it
is
one of the representations
\begin{equation}
\label{CartanRemark}
\pi_n^{\nu}: S^m((\C^{n})),\, \, 
S^m((\C^{n})'),\,  \,
\bigwedge^j(\C^n), \,\, 
\bigwedge^j((\C^n)'),\, \, 
m>0, \, \, 1\le j\le n-1.
\end{equation}
See \cite[Theorem 4.6.3]{Howe}.

We start with
the branching
rule of $GL(n)$-representations under
$GL(1)\times GL(n-1)$,
realized as block-diagonal
matrices in  $GL(n)$.
The following lemma
is a
consequence
of the known branching rule 
for $(\pi_n^\mu, GL(n))$ under $GL(n-1)$;
see e.g. \cite[Theorems 8.1.1, 8.1.2]{good}.

\begin{lemma}
\label{geqgeqlemma}
Consider the irreducible representation 
$(GL(n), \pi_n^{\mu})$
 under $GL(1)^2\times GL(n-2)$. \begin{enumerate}
    \item 
If there exist 
two gaps in the highest weight $\mu$, i.e., there exist
$1<i < j\le n$
such that 
$\mu_{i-1}
>\mu_{i}
\ge \mu_{j-1}>
\mu_{j}
$, 
then the branching 
 to $GL(1)^2 \times GL(n-2)$ is not multiplicity-free.

\item If $\mu$ has only one gap, i.e.,
$$\mu =(\overbrace{l_1+l_2,\cdots, l_1+l_2}^{{j }},\overbrace{l_2,\cdots,l_2}^{{n-j }}),$$ 
where $2 \leq j \leq n-2$, 
$l_1 \geq 2$,
$l_2\ge 0$ and $n \geq 4$,
then the branching  of
$\pi_{n}^{\mu} $ to $GL(1)^2  \times GL(n-2)$ is multiplicity-free, but its branching  to $GL(1)^3 \times GL(n-3)$ is not multiplicity-free.
\end{enumerate}
\end{lemma}

\begin{proof} (1) We apply  \cite[Theorems 8.1.2]{good}. 
When  $\mu_{i-1}
>\mu_{i} \ge \mu_{j-1}
>\mu_{j}$, the 
$(n-1)$-tuples (changing
the $i-1$
respective $j-1$ components and dropping
the component $\mu_{i}$ in $\mu$)
$$\nu_1 = (\mu_1, \cdots,
\mu_{i-1} - 1,
\mu_{i+1}, \cdots, \mu_{n})$$ and $$ \nu_2 = (\mu_1, \cdots, \mu_{i-1}, \mu_{i+1} \cdots, \mu_{j-1} - 1,
\cdots, \mu_{n} )$$ are the only ones
interlacing  $\mu$,  also (dropping $\mu_{j}$), $$\lambda = (\mu_1, \cdots, \mu_{i-1} -1, \mu_{i+1} \cdots, \mu_{j-1}-1, \mu_{j+1}, \cdots, \mu_{n})$$
interlaces both $\nu_1$ and $\nu_2$.
Therefore, the representation
$$g=\mbox{diag} (z_1, z_2, A) \mapsto z_1^{\mu_{i} + 1} z_2^{\mu_{j} + 1} \pi^{\lambda}_{n-2}(A), g\in GL(1)^2\times  GL(n-2) $$
occurs with multiplicity at least two.

Note first that
the 
$(n-1)$-tuples
 $$\nu_1(k) =(\overbrace{l_1+l_2,\cdots, l_1+l_2}^{j-1},l_2+k,\overbrace{l_2,\cdots,l_2}^{n-1-j }), \  0 \leq k \leq l_1$$
 are the only possible Young diagrams
 interlacing $\mu$;  for  each fixed $\nu_1(k)$,
 $$\nu_2= (\overbrace{l_1+l_2,\cdots, l_1+l_2}^{j-2 },l_{2}+s+k,l_2+t,\overbrace{l_2,\cdots,l_2}^{n-2-j }), 0\leq t \leq k, 0\leq s\leq l_1-k$$ 
are the only possible
ones interlacing $\nu_1(k)$; then
  $$\nu_3= (\overbrace{l_1+l_2,\cdots, l_1+l_2}^{{j-3  }},l_{2}+s+k+p,l_2+t+q,l_2+u,\overbrace{l_2,\cdots,l_2}^{{n-3-j}}),$$$$ 0\leq u \leq t, 0\leq q\leq k+s-t,  
  0\leq p\leq l_1-k-s$$
  are the only 
  ones interlacing $\nu_2$.
  Consequently 
the $GL(1)^2 \times GL(n-2)$-
 representations
$$
g=\mbox{diag}(
z_1,  z_2, A) 
\mapsto z_1^{l_1+l_{2}-k} z_2^{l_1+l_{2}-s-t} \pi^{\nu_2}_{n-2}(A), g\in GL(1)^2\times GL(n-2) $$
are 
the only subrepresentations
and they occur with multiplicity one.

However the following representation of $GL(1)^3 \times GL(n-3)$  occurs with multiplicity at least two,
$$
g=\mbox{diag}(
z_1,  z_2, z_3, B) 
\mapsto z_1^{l_{2}+1} z_2^{l_2 +1} z_{3}^{l_1+l_2-1}\pi^{\nu_3}_{n-3}(B),
g\in GL(1)^3\times GL(n-3),$$ where
$$ \nu_3 = (\overbrace{l_1+l_2,\cdots, l_1+l_2}^{{j-2  }},l_1+l_2-1,\overbrace{l_2,\cdots,l_2}^{{n-2-j}}),$$
since it can be 
obtained from 
the branching 
of $\pi_n^{\mu}$
by two
different
interlacing
procedures.\end{proof}

The next lemma follows from Lemma \ref{geqgeqlemma}, \cite[Theorem 4.6.3]{Howe} and \cite[Theorems 8.1.1, 8.1.2]{good}.

\begin{lemma}
\label{geqlemma}
Consider the branching of
$(GL(r+b), \pi_{r+b}^{\mu})$
under $GL(1)^r\times GL(b)$.
\begin{enumerate}
\item $r=1$. 
The branching  of $\pi_{1+b}^{\mu} $ to $GL(1)  \times GL(b)$ is always multiplicity-free.

\item $r=2$. If the branching  
of $\pi_{2+b}^{\mu} $ to $GL(1)^2  \times GL(b)$ ($b \geq 2$) is multiplicity-free, then $\mu$ is of the form $$\mu =(\overbrace{l_1+l_2,\cdots, l_1+l_2}^{{j }},
\overbrace{l_2,\cdots,l_2}^{{b+2-j }}),$$ where $1 \leq j \leq b+1$ and $l_1 \geq 0$.

\item $r\geq 3$. The branching  of $\pi_{r+b}^{\mu} $ to $GL(1)^r  \times GL(b)$ is multiplicity-free 
if and only if
up 
to a character it is
the exterior representation $\bigwedge^j(\C^{r+b})$, the symmetric representation $S^m(\C^{r+b})$ or their dual.

Moreover, (1), (2) and (3) hold when restricting $\pi_{r+b}^{\mu}$ to 
$$\{(a_1, \dots, a_r,A)\in GL(r+b) \ | \ a_1^2 \dots a_r^2 \det(A) = 1 \}.$$
\end{enumerate}
\end{lemma}

\begin{proof}
The results in (1), (2) and (3) about  branching from $GL(r+b)$ to $GL(1)^r \times GL(b)$ follow from Lemma \ref{geqgeqlemma}, \cite[Theorem 4.6.3]{Howe} and \cite[Theorem 8.1.2]{good}. For the restriction from $GL(r+b)$ to 
$\{(a_1, \dots, a_r,A). \ | \ a_1^2 \dots a_r^2 \det(A) = 1 \}$ part (1) follows from \cite[Theorem 8.1.1]{good} and part (2) follows by the proof of Lemma \ref{geqgeqlemma}. Part
(3) follows from Lemma \ref{geqgeqlemma} and an analysis of the action of the Cartan subgroup on the eigenvectors of the symmetric tensor powers and the exterior powers.
\end{proof}



 We now  consider the
 representations
 of the product $GL(r+b)\times G(r)$
 under $M^{\mathbb C}$.

\begin{lemma}
\label{classthmglrbglr}
Consider the branching
of $\pi:=(
GL(r+b)\times GL(r), \pi_{r+b}^{\mu}\otimes \pi_{r}^{\nu}) $
under $M^{\mathbb C}$.
\begin{enumerate}
\item For $r=1$ the branching is multiplicity-free for any representation of $GL(1+b) \times GL(1)$.
    \item For $r=2$ the branching
is multiplicity-free
if and only if
 $$\mu =(\overbrace{l_1+l_2,\cdots, l_1+l_2}^{{j }},\overbrace{l_2,\cdots,l_2}^{{b+2-j }}),$$ where $1 \leq j \leq b+1$ and $l_1 \geq 0$ and $\pi_{r}^{\nu}$ is a character or  $\pi_{r+b}^\mu$ is a character and $\pi_r^{\nu}$ is, up to a character, one of the exterior representations, the symmetric representations or their duals.
\item For $r\geq 3$ the branching 
is multiplicity-free if and only if  one of $\pi_{r+b}^\mu$ and $\pi_r^\nu$ is a character and the other is, up to a character, one of the exterior representations, the symmetric representations or their duals.
\end{enumerate}
\end{lemma}

\begin{proof}
It follows from Lemma \ref{geqlemma} and (\ref{CartanRemark}) that all these representations are multiplicity-free when restricted to $M^{\mathbb C}$, we only prove the converse parts.

Part (1) follows from \cite[Theorem 8.1.1]{good}. For (2) and (3), suppose it is multiplicity-free. 
It follows from Lemma \ref{geqlemma} that
$\pi_{r+b}^\mu$ is
of the form
$$
\mu=(\overbrace{l_1+l_2,\cdots, l_1+l_2}^{{j }},\overbrace{l_2,\cdots,l_2}^{{r+b-j }})$$
and $\pi_r^\nu$ is
a character or
of the form
(\ref{CartanRemark}). 
Suppose $\pi_r^\nu$ is
not a character.
Then we prove that
 $\pi_{r+b}^{\mu}$ must be a character.

As in the proof of Lemma \ref{geqgeqlemma}, when restricting $\pi_{r+b}^{\mu}$
to $GL(1) \times GL(1) \times GL(r+b-2)$ there are
the subrepresentations
$$(z_1,z_2,A) \mapsto z_1^{l_2} z_2^{l_1+l_2} \pi_{r+b-2}^{\mu'}(A)$$
acting on the  $V_1$, and
$$(z_1,z_2,A) \mapsto z_1^{l_2+1} z_2^{l_1+l_2-1} \pi^{\mu'}_{r+b-2}(A)$$
acting on $V_2$, where $\mu' = (\overbrace{l_1+l_2,\cdots, l_1+l_2}^{{j-1}},\overbrace{l_2,\cdots,l_2}^{{b+r-j-1 }})$. 

Case 1: Up to a character,
  $\pi_{r}^{\nu} $
is the symmetric representation or its dual.
It suffices to 
consider $\pi_r^{\nu}$ being the symmetric representation. Let  
$D=\mbox{diag}(z_1,\cdots, z_r)\in GL(1)^r$
and consider the following 
two actions 
on $V_1 \otimes \mathbb{C}(x_1^{m-1}x_2)$ and
on  $V_2 \otimes \mathbb{C}x_1^{m}$,
$$
g=\mbox{diag}
(D, A, D):
{v_{1}\otimes (x_1^{m-1}x_2)}\mapsto  
z_{1}^{l_{2}+m} z_{2}^{l_1+l_{2}} [\pi^{\mu'}_{r+b-2}(A)v_{1}\otimes 
x_{1}^{m}]
$$
and
$$
\mbox{diag}
(D, A, D):{v_{2} \otimes x_{1}^{m}} \mapsto z_{1}^{l_{2}+m} z_2^{l_1+l_{2}} [\pi^{\mu'}_{r+b-2}(A)v_2\otimes 
(x_{1}^{m-1} x_2)].
$$
These are equivalent
representations 
of $M^{\C}$ 
in $\pi$, when $r \geq 2$.
The case $r = 1$ is already covered
by the prvious Lemma.




Case 2: Up to a character $\pi_{r}^{\nu} $ is  the exterior  representation $\wedge^{j}(\mathbb{C}^{r})\cong \mbox{span}_{\mathbb{C}}\{e_{i_1}\wedge \cdots \wedge e_{i_j} \ | \ 1 \leq i_{1}< i_2 < \cdots < i_j < r\}$, $1 < j < r$. Let  $D=\mbox{diag}(z_1, \cdots, z_r)\in GL(1)^r$,
$v_1\in V_1, v_2\in V_2$ and
consider the  
$M^{\C}$-representations
on
$V_1\otimes \mathbb{C}(e_{1} \wedge e_{3} \wedge \cdots \wedge e_{j+1})
$,
$$\mbox{diag}
(D, A, D) \mapsto 
 z_1^{l_2} z_2^{l_1+l_{2}}
 (z_1 \prod_{i=3}^{j+1}z_{i})[\pi^{\nu_2}_{b+r-2}(A)] 
$$
and on 
$V_{2}\otimes 
\mathbb C(e_{2} \wedge e_{3} \wedge \cdots \wedge e_{j+1})$
$$\mbox{diag}
(D, A, D)\mapsto 
z_1^{l_{2}+1} z_2^{l_1+l_{2}-1}
(\prod_{i=2}^{j+1} z_i)[\pi^{\nu_2}_{b+r-2}(A)v_{2} ].
$$
They are isomorphic 
and occur with multiplicity
at least two.
\end{proof}

Theorem \ref{KresThm} (1)
follows from Lemma \ref{classthmglrbglr}.

\subsection{
$(G, K)=(Sp(n,\mathbb{R}), U(n))$, $(Mp(n,
\mathbb{R}), \widehat{U(n)})$}
Here $(Mp(n,
\mathbb{R}), \widehat{ U(n)})$ is a double
cover of  $U(n)$
under the covering
$Mp(n, \mathbb R)\to Sp(n, \mathbb R)\supset U(n)$.
We consider first $G=Sp(n,\mathbb{R}).$
The Cartan decomposition is $\mathfrak{sp}(n, \mathbb{R})=\mathfrak k\oplus \mathfrak p$
where $\mathfrak p=\{X\in M_{n, n}(\mathbb C); X^t=X \}$ as real space and
the adjoint of $K=U(n)$  on $\mathfrak p$ is 
$A\in U(n): X\mapsto
AXA^t$.
A maximal abelian subalgebra of $\mathfrak p$ is given by
$\mathfrak{a}=
\sum\limits_{i=1}^{n}\mathbb{R}E_{ii}.$
The group $M=C_{K}(\mathfrak{a})$ is
 $$M=\{D  \ | \ D \mbox{ is \ real \ diagonal \ and } \ D^{2}=I_n \}\cong \mathbb{Z}_{2}^{n}.$$

We will now prove  Theorem \ref{KresThm} (2).

\begin{proof} Since $M=\mathbb{Z}_{2}^{n}\subseteq U(1)^{n} \subseteq U(n)$, for the branching  $(U(n),\pi_{n}^{\mu})|M$ to be multiplicity-free $\pi_{n}^{\mu}$ must be as in the list (\ref{CartanRemark}), and we treat each one accordingly.

Case 1:   $\pi_{n}^{\mu}$  is 
the exterior representation, i.e. $\mu= (l, \dots, l, l-1, \dots, l-1)$ for $l \in \Z$.
For any $A = \mbox{diag}\{a_1, \cdots, a_n\}\in \Z_2^n \subseteq U(n)$, write $A=(\prod_{i=1}^{n}a_i)^{\frac{1}{n}}A_1$ with $A_1\in SU(n)$. 
$$\pi_n^{\mu}(A_1)(e_{i_1} \wedge \dots \wedge e_{i_j})
=(\prod_{i=1}^{n}a_i)^{\frac{-j}{n}}\prod_{k=1}^{j}a_{i_{k}}.$$
Hence,
$$\pi_n^{\mu}(A)({e_{i_1} \wedge \dots \wedge e_{i_j}} )=[(\prod_{i=1}^{n}a_i)^{\frac{|\mu|-j}{n}}\prod_{k=1}^{j}a_{i_{k}}]e_{i_1} \wedge \dots \wedge e_{i_j}.
$$
Any irreducible representation of an abelian Lie group is a character. 
Consider the  
following character of $M$,
$$\chi_{i_1, \cdots,i_j}:M\rightarrow \mathbb{C};\mbox{diag}(a_1,\cdots,a_n)\mapsto (\prod_{i=1}^{n}a_i)^{\frac{|\mu|-j}{n}}\prod_{k=1}^{j}a_{i_{k}}.
$$ For any distinct pair 
$1 \leq i_1 < i_2 < \dots < i_j\leq n$, $1 \leq k_1 < k_2 < \dots < k_j\leq n$ there exists an index $i_s \not \in \{k_1, \cdots, k_j\}$,  such that $$\chi_{i_1, \cdots,i_j}(I_n-2E_{i_s,i_s})\neq \chi_{k_1, \cdots,k_j}(I_n-2E_{i_s,i_s}),$$ and thus $\chi_{i_1, \cdots,i_j}\neq \chi_{k_1, \cdots,k_j}$.
This case is proved.

Case 2:  
  $\pi_{r}^{\nu} $
is the symmetric representation or its dual.
For $A=\mbox{diag}(a_1, \cdots, a_n), a_i \in \mathbb{Z}_2$, if $m$ is even, the characters $$\pi^{\mu}(A)(x_1^m )=a_1^{m},\,  \pi^{\mu}(A)(x_2^m )=a_2^{m}, 
$$ are the  same; if $m$ is odd, the characters $$\pi^{\mu}(A)(x_1^m )=a_1^{m}, \pi^{\mu}(A)(x_1x_2^{m-1} )=a_1a_2^{m-1}
$$ are the same when restricted to $\mathbb{Z}_2^n$. The similar claim holds for any $(l_2+l_1,l_2,\dots,l_2)$ and $l_1 \geq 2$ and the dual representation.
So the symmetric tensor powers are excluded.

Now if $(G, K)
=(Mp(n,
\mathbb{R}), 
\widehat{ 
U(n)})
$, then we may choose
a realization for the double cover 
$\widehat{ 
U(n)}$ 
as  $
\widehat
{U(n)}
=\{(g, z); g\in U(n), z^2=\det g \}$ with 
the corresponding $\widehat M
=\{(g, z); g\in M, z^2=\det g\}$. There is then a corresponding double cover $\widehat{U(1)^n}$,
$\widehat M\subset 
\widehat{U(1)^n}
\subset  
\widehat{ 
U(n)}$, and the same arguments as
above prove our claim, 
noticing that one dimensional representations of $\widehat{U(n)}$, $(g, z)\to 
z^{2m+1}$ cannot
be defined on 
$U(n)$.

\end{proof}

\subsection{$(G,K) = (SO^*(2n), U(n))$}
The 
Cartan decomposition
is $ \mathfrak{so}^*(2n) 
=\mathfrak k +\mathfrak p$
with $\mathfrak p$ being
realized as $\mathfrak p
= \{X\in M_{n, n}(\mathbb C), X^t=-X\}$
 as a real space and  
 $K$ acts on
 $\mathfrak p$
 as in the $Sp(n, \mathbb R)$-case,
 $A\in K: X\mapsto AXA^t$.
 We choose the Cartan subalgebra in $\mathfrak p$ as
$$
 \af = \sum\limits_{j=1}^{[\frac{n}{2}]} 
 \R (E_{2j-1,2j} - E_{2j,2j-1}).
 $$
The group $M=C_{K}(\mathfrak{a})$ for  $SO^*(4r)$ is 
$$M=SU(2)^{r}\subset U(2r),$$
and for $SO^*(4r+2)$ it is
$$M=SU(2)^r \times U(1)$$
as  diagonal block matrices.



We now prove Theorem \ref{KresThm} (3).

\begin{proof} 
We shall prove the necessary and sufficient part
in the same time
and recall again that we disregard any factor of characters in the representations of $K$.

Let $n = 2r$ be even.  First we note that the symmetric representations and its dual do split multiplicity-free; if $\pi_{2r}^{\mu}$ is the symmetric representation then
$$S^m(\mathbb{C}^{2r}) \cong \sum_{i_1 + \dots + i_r = m} S^{i_1}(\C^2) \otimes \cdots \otimes S^{i_r}(\C^2), $$
which is multiplicity-free as a $SU(2)^{r}$-representation. The same holds for the dual. 

Case 1:  $r>3$.
By \cite[Corollary 1.2A]{St}
we have
that the only possible candidates 
for  $\mu$ is trivial, the exterior representation, the symmetric representation or its dual 
The symmetric representation and its dual do split multiplicity-free. If $\pi_{2r}^{\mu}$ is the representation on $\bigwedge_{i=1}^n \C^{2r}$ with $n < 2r$  even, then the two representations
$$V_1=\mathbb{C}[(e_1 \wedge e_2) \wedge 
 \cdots   \wedge (e_{n-3} \wedge e_{n-2}) \wedge (e_{n-1} \wedge e_{n})]$$
and
$$V_2=\mathbb{C}[(e_1 \wedge e_2) \wedge  \cdots \wedge  (e_{n-3} \wedge e_{n-2}) \wedge (e_{n+1} \wedge e_{n+2})]$$
are both trivial representations for $SU(2)^r$. 

If $\pi_{2r}^{\mu}$ is  $\bigwedge_{i=1}^n \C^{2r}$ with $n<2r-1$ odd, then the representations generated by
$$(e_1 \wedge e_2) \wedge \cdots \wedge(e_{n-2} \wedge e_{n-1}) \wedge e_{n+1}
$$
and
$$(e_1 \wedge e_2) \wedge \dots \wedge (e_{n+2} \wedge e_{n+3}) \wedge e_{n+1}$$
are isomorphic. We conclude that the exterior representations do not split multiplicity-free.

Case 2:  $r=3$.
The only candidates are $\mu$ is trivial, the exterior representation,  the symmetric representation  and its dual or $2 \omega_3$, $m \omega_2$ or $m \omega_{4}$ \cite[Corollary 1.2A]{St}.

 For $m \omega_2$ and $m \omega_4$,  we see that  $(e_1 \wedge e_2)^{ \otimes m}$ and $(e_3 \wedge e_4)^{\otimes m}$ are in $V_{m \omega_2} \subseteq V_{\omega_2}^{\otimes m}$. These give copies of the trivial representation, so this does not split multiplicity-free. The same holds for $m \omega_4$. For $2 \omega_3$ we get the same representations by $(e_1 \wedge e_2 \wedge e_3)^{\otimes 2}$ or $(e_4 \wedge e_5 \wedge e_6)^{\otimes 2}$.

Case 3:  $r=2$. The only candidates are $\mu= a \omega_1 + b \omega_2$, $\mu=a \omega + b \omega_3$ or $\mu=a \omega_2 + b \omega_3$ \cite[Corollary 1.2A]{St}.

Recall the restriction formula \cite[Equation 9.24]{good}
$$
\mbox{Res}^{\mathrm{GL}(4)}_{\mathrm{GL}(2)\times \mathrm{GL}(2) }(\pi_{4}^{\lambda})=
\oplus_{\mu, \nu}
c_{\mu,\nu}^{\lambda}\pi_2^{\mu}\otimes \pi_2^{\nu},$$
where the sum is
over those $\mu$ and $ \nu$
of lengths at most $2$
and such $\mu\subset \lambda$ that 
that $\nu$ is partition of $|\lambda|-|\mu|$
and  the multiplicities $c_{\mu,\nu}^{\lambda}$ are nonnegetive integers named Littlewood-Richardson coefficients \cite[Section 9.3.5]{good}. A direct calculation  shows that for $V_{(\mu_1,\mu_2,0)} (\mu_2\neq 0)$, $V_{(0,\mu_2,\mu_3)}(\mu_2\neq 0)$ and  $V_{(\mu_1,0,\mu_3)}(\mu_1,\mu_3\neq 0)$ there is multiplicity when restricting to $SL(4,\C)$. This proves the theorem for $n$ even.

Let  $n = 2r+1$
be odd. Then
our claim is a consequence
of 
\cite[Theorem 8.1.2]{good}
by realizing $M$
as a subgroup $M=SU(2)^r \times U(1) \subseteq SU(2r) \times U(1) \subseteq  U(2r+1)$.
\end{proof}

\subsection{
$(G, K) =
(SO_{0}(n,2), SO(n)\times SO(2))
$, 
$(Spin_0(n,2), Spin(n)\times SO(2)/\{(1,1), (-1,-1)\})$
}
The space $\mathfrak p
=M_{n, 2}(\mathbb{R})$
and $K$ acts on 
$\mathfrak p$ by $(A, D):
X\mapsto AXD^{-1}$.
A maximal abelian subalgebra of $\mathfrak p$ is given by
$$\mathfrak{a}=
\mathbb R E_{1,1}+
\mathbb R E_{2,2}.$$
The group $M=C_{K}(\mathfrak{a})$ is
$$M = \{ (A, D)\ | A=\text{diag}(A_1, B), 
A_1=D=\pm I_2, B \in SO(n-2) \} \cong SO(n-2)
\times \mathbb Z_2.$$
For the double cover $Spin(n,2)$ we obtain the subgroups via the covering $Spin(n, 2) \to SO_0(n, 2)$. 
We notice again here that
we can disregard
the character factor and  focus on representations of $Spin(n)$.

The set
$$\Gamma_{2k+1}=\{ \mu | \mu_1 \geq \mu_2 \cdots \geq \mu_k \geq 0; 2\mu_i\in \mathbb{Z}, \mu_i-\mu_j \in \mathbb{Z}\}$$
parametrizes the irreducible representations of $\mathfrak{so}(2k+1,\C)$ and
$$\Gamma_{2k}=\{ \mu | \mu_1 \geq \mu_2 \cdots \geq |\mu_k| ; 2\mu_i\in \mathbb{Z}, \mu_i-\mu_j \in \mathbb{Z}\}$$
 the irreducible representations of $\mathfrak{so}(2k,\C)$. For $\mu_i\in \mathbb Z$ being integers these are representations of $SO(n)$,
the others  are
 representations
of the spin groups; see e.g. \cite[Chapter 19, Chapter 20]{FulHar}. 
We have  the following branching rules \cite[Section 25.3]{FulHar}:
\begin{equation}
\label{res-so-so}
\mbox{Res}_{\mathfrak{so}_{2k}(\C)}^{\mathfrak{so}_{2k+1}(\C)}(\pi_{2k+1}^{\mu})=\bigoplus_{\xi} \pi_{2k}^{\xi},
\end{equation} 
    where the sum is over all $\xi\in\Gamma_{2k} $ such that  $\mu_1 \geq \xi_1 \geq \dots \geq \mu_k \geq |\xi_k|$, and
\begin{equation}
\label{res-so-so2} \mbox{Res}_{\mathfrak{so}_{2k-1}(\C)}^{\mathfrak{so}_{2k}(\C)}(\pi_{2k}^{\xi})=\bigoplus_{\eta} \pi_{2k-1}^{\eta}
\end{equation}
    where the sum is over all $\eta \in \Gamma_{2k-1} $ such that $\xi_1 \geq \eta_1 \geq \dots \geq \eta_{k-1} \geq |\xi_k|$ \cite[Theorem 19.22, Section 25.3]{FulHar}. We now prove Theorem \ref{KresThm} (4).
    
\begin{proof} We first check the $(SO_0(n,2),SO(n) \times SO(2))$ case. As $M=
SO(n-2)
\times \Z_2 
\subset K=SO(n)\times SO(2)$
the multiplicity-free condition
under $M$ implies 
the same for
    $SO(n)$-representations under $SO(n-2)$ using the branching rules in \cite{koike}; this shows that the $\Z_2$ factor always acts in the same way on each subspace.
    
Let  $n=2k+1$ be odd.
    The case $k=1$ is obvious.
   For $k>1$
   it follows
   from  the
above
restriction formula that this is multiplicity-free if and only if $\pi_{2k+1}^\mu
   =\pi_{2k+1}^0   $ is the trivial representation.
Let  $n=2k$ be even. 
It follows again
that the restriction is multiplicity-free if and only if 
   $\mu=(a,\dots,a,\pm a)$
   for some $a\ge 0$.

Now we consider $G=Spin_0(n,2)$, the
connected component 
of the double cover of
$SO_0(n, 2)$. Let $\pi: Spin_0(n, 2)\to SO_0(n, 2)$ be the covering. The maximal compact subgroup 
$K\subset G$ is then $K=
Spin(n)\times SO(2)/\mathbb Z_2$
with $\mathbb Z_2=\{(1, 1), (-1, -1)\}$, $\pm 1$ being
the central elements in $Spin(n)$ and in $SO(2)$, respectively.
We use  the construction
  of $Spin(n)$ as in \cite{FulHar}
 and of their representations along
 with the notation there. First we 
determine  the corresponding subgroup $M=C_K(\mathfrak a)\subset K$. Note that $\pm 1 \in Z(Spin(n))$, the center of $Spin(n)$, where $\{ \pm 1\} = \pi^{-1}(I)$, and $\pi^{-1}(SO(n-2)) \cong Spin(n-2)$. 
 For $Spin(2k)$ let $v_{1} = \frac{e_1 + i e_2}{\sqrt{2}},\dots, v_k = \frac{e_{2k-1} + i e_{2k}}{\sqrt{2}}$ and $v_{n+j} = \overline{v_{j}}$;
 for $Spin(2k+1)$, $v_{2k+1} = e_{2k+1}$. The group ${M}\subset 
Spin(n)\times SO(2)/
\mathbb Z_2$ 
 is generated 
 in $Spin(n)\times SO(2)$, modulo $\mathbb Z_2$,
 by $\omega_1 = \frac{i v_1 \cdot v_{n+1} - i v_{n+1}v_1}{2}$ and the base $B \backslash \{v_1,v_{k+1} \}$ in 
 $CL_n$,
 and the group generated by $B \backslash \{v_1,v_{k+1} \}$ is equal to $Spin(n-2)$.
Here $CL_n$ is the
Clifford algebra of $\mathbb R^n$. The element  $\omega_1$ generates a $\Z_4$ subgroup and commutes with $Spin(n-2)$.
Thus $M=Spin(n-2)\times 
\mathbb Z_4/\mathbb Z_2$.

Now irreducible
presentations
of $K$ define
representations
of $Spin(n)\times SO(2)$ and the 
factor $SO(2)$
acts as scalars,
so we can disregard
this factor.
We prove that
 a 
representation $\tau$ of $Spin(n)$ 
has the multiplicity-free
property under  $M$ if and only if

(1) $\tau=(a,\dots,a,\pm a)$ when $n = 2k$ is even,

(2) $\tau=(\frac{1}{2},\dots,\frac{1}{2})$ when $n = 2k+1$ is odd.

 Let $\tau$ be an irreducible representation of $Spin(n)$
 and
 $$ V_{\tau}|_{Spin(n-2)} = \bigoplus_i V_i$$
its decomposition under $Spin(n-2)$.
We can choose the decomposition so that $\Z_4$ acts as a character on each $V_i$ by the commutativity. We also see that $\omega_1^2 = -1$ is a center element
in $Spin(n)$, and  there are only two options for the character of $\Z_4$ on $V_i$.

Now let $n$ be even and $\tau$ a $Spin(n)$-representation. Let $\omega_j = \frac{i v_j \cdot v_{n+j} - i v_{n+j}v_j}{2}$ and $\omega = \omega_1 \dots \omega_n$ the central element in $Spin(n)$. Then
$$\omega_1 = (-1)^{n-1} \omega \cdot \omega_2 \dots \omega_n.$$
We now observe that $\omega_2 \dots \omega_n$ is the central element in $Spin(2n-2)$ and acts as a character determined by the equivalence class of $V_i$, and $(-1)^{n-1} \omega$ acts as a character determined by $\tau$. Hence the character of $\omega_1$ is uniquely determined by the conjugacy class $\tau_i$ and $\tau$, thus the restriction to $M$ is multiplicity-free if and only if $\tau|_{Spin(2n-2)}$ is multiplicity-free. It follows from the branching rule (\ref{res-so-so}) and (\ref{res-so-so2})
that this is the case
only for $\tau=\pi_{n}^{\mu}$ where $\mu = (a,\dots,a,\pm a)$ for some $a \geq 0$ a half-integer.

Now let $n$ be odd. Let $\tau
=\pi_n^\mu$ be an irreducible representation. Clearly if the multiplicity of a representation in $\tau|_{Spin(n-2)}$ is higher than two the restriction to $M$ is not multiplicity-free because there are only two options for the character of $\mathbb{Z}_4$. It follows from Equations (\ref{res-so-so}) and (\ref{res-so-so2}) that
$$\mu = (1,\dots,1,0,\dots,0)$$
or
$$\mu = (\frac{1}{2},\dots,\frac{1}{2}).$$
From \cite[Theorem 3.2]{koike} we see that $(1,\dots,1,0,\dots,0)$ only splits multiplicity-free for the trivial representation. Now let $\mu = (\frac{1}{2},\dots,\frac{1}{2})$; this is the spin representation given on $V = \bigwedge^{\bullet} W$, where $W$ is spanned by $v_1, \dots, v_n$ \cite[Chapter 20]{FulHar}. We see that $V|_{Spin(n-2)} = V_1 \oplus V_2$ where both are isomorphic to the spin representation. By direct computations we see that $\omega_1$ acts as $i$ on the one and as $-i$ on the other, thus giving two different 
and being free. Hence $\tau|_M$ splits multiplicity-free only for the trivial representation and the spin representation.
\end{proof}

\subsection{$(G,K) = (E_{6(-14)},Spin(10)U(1))$}

The group
$E_{6(-14)}$ 
is defined here as
the real adjoint Lie group of rank $2$ by its action on 
the Lie algebra
$\mathfrak e_{6(-14)}$.
To describe
the subgroup $M$
we use the theory of Jordan triples
\cite{loos} 
to describe
the action of the 
maximal compact 
subgroup $K$
on  $\mathfrak p$
and $\mathfrak p^+$.
The strongly orthogonal
roots $\{\gamma_1, \gamma_2\}$
determine a frame $\{e_1, e_2\}
\subset \mathfrak p^+
$
of minimal tripotents
in $\mathfrak p^+$ 
as root vectors, 
and $\mathfrak a
=\mathbb R\Re e_1
+\mathbb R\Re e_2\subset \mathfrak p
$  is the corresponding Cartan subalgebra.
The root space (also 
called Peirce decomposition in Jordan triples)
of $\mathfrak p^+ $
with respect to the co-roots
$\{H_{\gamma_1}, H_{\gamma_2}\}$
of $\{\gamma_1, \gamma_2\}$
   is (see e.g. \cite{dozha})
   \begin{equation}
   \label{peirce-e6}
   \mathfrak p^+ 
   =(V_{11}
   \oplus 
   V_{12}
   \oplus
      V_{22})
      \oplus 
      (V_{10}
      \oplus 
      V_{20})=
   (\mathbb Ce_1 \oplus 
   \mathbb C^6 \oplus
  \mathbb Ce_2 )
\oplus( \mathbb C^4 \oplus
 \mathbb C^4)
\end{equation} 
with roots of
$(H_{\gamma_1}, H_{\gamma_2})$
being $(2, 0), (1, 1),
(0, 2),  (1, 0),
(1, 0)
$ on the respective spaces.
Let $\mathfrak m_1 = \mathfrak m\cap \mathfrak k'$. Then $\mathfrak m_1^{\C} 
=\mathfrak{spin}(6,\C)
$ is the subalgebra of $\mathfrak k'^{\C}=\mathfrak{spin}(10,\C)$
preserving the decomposition  (\ref{peirce-e6}) above; see \cite[Appendix C]{knapp}. Let 
$ 
H= i(H_{\gamma_1}
+H_{\gamma_2})- 2Z
$,
where $Z$ is the center element of $\mathfrak k$ so that $ad(Z) = i$ on $\mathfrak p^+$.
Then $ad (H) e_1 =ad (H) e_2=0$, so $H\in \mathfrak m$. We have 
$\mathfrak m=
\mathfrak m_1 \oplus \mathbb RH, \mathfrak m_1 
= \mathfrak{spin}(6)
$; see also \cite{knapp}.

We claim that $K=  Spin(10) \times U(1) / \Gamma$ 
as groups acting on $\mathfrak p^+$, where $\Gamma=\mathbb Z_4$ is 
described below in (\ref{K-des}).
The group  $K_1$, the analytic subgroup with Lie algebra $\mathfrak{spin}(10)$, is simply connected and thus equal to $Spin(10)$ \cite[Corollary 3.2]{schlichtkrull}.
Thus the maximal compact subgroup 
$K=  Spin(10) \times U(1) / \Gamma$,
where $U(1) = \{ e^{it Z} \ | t \in \R \} $ 
and $\Gamma = \{ (a,b) \in Spin(10) \times U(1) \ | \ ab = 1 \}$.
 Also, $\p^+$ is the half-spin representation of $\mathfrak{spin}(10)$  \cite{dozha} and $Spin(10,\C)$ acts faithfully on $\p^+$ \cite[Lemma 20.9, Proposition 20.15]{FulHar}. Then so does $Spin(10)$ \cite[Chapter 5.4]{var}. So the only elements of $Spin(10)$ which act as scalars on $\p^+$ are in the center $\mathbb{Z}_4
 =\{1, \omega, -1, -\omega\} 
 $ \cite[Chapter 5.5]{var},  which 
 can further be identified  with $
 \mathbb{Z}_4
 \cong
 \{1,i,-1,-i\}
 \subset Spin(10) $
 by its defining action on $\mathfrak p^+$. It follows by similar arguments as for the proof of \cite[Theorem 3.10.7]{yokata} that
 \begin{equation}
    \label{K-des}
  \Gamma \cong \mathbb Z_4,
 \Gamma= \{(1,1), (i,e^{- \frac{Z}{2} \pi}), (-1,e^{ Z \pi}), (-i,e^{- \frac{3Z}{2} \pi}) \}
 \subset
Spin(10)\times U(1).
\end{equation}
This proves our claim on $K$.

 The Lie group $M$ is connected \cite[Lemma 4.2]{schlichtkrull},  each element $m \in M$ 
 preserves the 
 frame $\{e_1, e_2\}$
 and thus 
 the  Peirce decomposition    (\ref{peirce-e6}), and 
 defines then an action $m|_{V_{10}}
\in U(V_{10}) = U(4)
$
in the 
Peirce component $V_{10}=\mathbb C^4$. We note that $V_{10}$, $V_{20}$ are two non-isomorphic spin representations of $\mathfrak{spin}(6) = \mathfrak{su}(4)$ by \cite{dozha}. As $SU(4)$ acts faithfully and $M_1 = M \cap K_1$ is connected \cite[Lemma 4.4]{schlichtkrull} it follows that $M_1 = M \cap K_1 \cong SU(4)$. Hence 
$$
M = SU(4) U(1) \cong SU(4) \times U(1) / \Gamma',$$
where $\Gamma' = \{ (a,b) \in SU(4) \times U(1) \ | \ a b|_{\p^+} = id \}$ and $U(1) = \{e^{t H} \ | \ t \in \R \}$. $H$ acts as $-i$ on $V_{10}$, $V_{20}$ and as $0$ on $V_{12}$. From
the exact roots 
given in \cite[Section 12]{EHW} it can be calculated that $V_{12}=\mathbb C^6$ is the defining representation of $\mathfrak{so}(6)=\mathfrak{spin}(6)
=\mathfrak{su}(4)$ and has kernel $\{ \pm 1 \} \in SU(4)$. 
Therefore the subgroup $\Gamma'$ is
$
\Gamma' = \{(1,1), (-1, e^{\pi (iD(e,e) - 2Z)} )
\}.
$

We prove Theorem \ref{KresThm} (5).

\begin{proof}The Lie algebra $\mathfrak{m} =
\mathfrak{spin}(6)
\oplus \mathbb RH=\mathfrak{m}_1 \oplus \mathbb{R} H$, and write $H$ in $\mathfrak
k=\mathfrak k'
\oplus \mathbb RZ$ as $H = H_1 + x Z$, $H_1\in 
\mathfrak k'$. We realize $\mathfrak g^{\mathbb C}$
as in  \cite[Section 12]{EHW}. The Lie algebra $\kf^{\C}$ is generated by the roots $\e_1 + \e_2$, $\e_{i+1} - \e_{i}$ $1 \leq i \leq 4$, and is isomorphic to a Lie algebra $\mathfrak{spin}(10,\C)$. We use
the notation
 in \cite{EHW} and 
 choose two
 strongly orthogonal roots
$$\gamma_1 = \frac{1}{2}(\e_1 - \e_2 - \e_3 - \e_4 - \e_5 - \e_6 - \e_7 + \e_8),
$$
$$\gamma_2 = \frac{1}{2}(-\e_1 + \e_2 + \e_3 + \e_4 - \e_5 - \e_6 - \e_7 + \e_8).$$
Then $\mathfrak{m}_1^{\C}$ is the Lie algebra generated by $\e_1 + \e_2$, $- \e_2 + \e_3$ and $-\e_3 + \e_4$. This is isomorphic to a Lie algebra $\mathfrak{spin}(6,\C)$, and corresponds to deleting nodes $1$ and $5$ in the Dynkin diagram $D_5$ of $\mathfrak{spin}(10,\C)$ in \cite{St}. Furthermore
$$- \frac{2}{3} \e_6 - \frac{2}{3} \e_7 + \frac{2}{3} \e_8,$$
is central and acts as $2$ on $\p^+$, so is equal to $-2iZ$. Also, $D(e,e) = h_{\gamma_1} + h_{\gamma_2} = - \e_5 - \e_6 - \e_7 + \e_8$. It follows that $-i H = -\e_5 - \frac{1}{3} \e_6 - \frac{1}{3} \e_7 + \frac{1}{3} \e_8$ and $- i H_1 = - \epsilon_5$. Now according to \cite{St} the only representations which split multiplicity-free when restricted to $\mathfrak{spin}(6,\C) \oplus \C^2 \subseteq \mathfrak{spin}(10,\C)$ are of the form $\{0,m \omega_1, m \omega_2, m \omega_5 \}$. Here $\C^2 = \C( \e_1 - \e_2 - \e_3 - \e_4) \oplus \C \e_5$.

We now investigate the
branching rule
of those $m \omega_i$, $i = 1,2,5$, under $\mathfrak{spin}(6,\C) \oplus \C \e_5$. For $\mu=m\omega_5$ we get the space of harmonic polynomials on $\mathbb C^{10}$. When restricting to $\mathfrak{spin}(6,\C)$ we get several trivial representations and the restriction is not multiplicity-free. 

Now we study $m \omega_1$ and $m \omega_2$. From \cite[Equation 25.34, Equation 25.35]{FulHar} we see
$$V_{m \omega_{1,2}}|_{\mathfrak{spin}(8,\C) \oplus \C \e_5} \cong \bigoplus_{i=0}^{m} V_{i \omega_1 + (m-i) \omega_2} \otimes \C_{j_i},$$
where $\C_{j_i}=\mathbb C$ is a representation of $\C \e_5$. When $m>1$ restricting this to $\mathfrak{spin}(6,\C)$ we get $(\frac{m}{2},\frac{m}{2},\frac{m}{2}-1)$ two times, so it is higher multiplicity. Now we study $\omega_{1,2}$. We follow the notation in \cite[Lecture 18]{FulHar} and \cite[Lecture 20]{FulHar}. Let $W = \langle e_1, \dots, e_5 \rangle \subseteq \C^{10}$ and $W' = \langle e_{6}, \dots, e_{10} \rangle $.
Now the representation labeled by $\omega_1$ is $S^+ = \bigwedge^{\mathrm{even}} W = \bigwedge^{\mathrm{even}} \C^5$. We restrict to $\mathfrak{spin}(8,\C) \oplus \C \e_5$ and call the fundamental weights of $\mathfrak{spin}(8,\C)$ $\theta_1, \dots, \theta_4$. We cut away the node corresponding to $\e_5 - \e_4$ from the Dynkin diagram of $\mathfrak{spin}(10,\C)$ and note that restricting on $S^+$ to $\mathfrak{spin}(8,\C)$ we get
$$ S^+ |_{\mathfrak{spin}(8,\C)} \cong \bigwedge^{\mathrm{even}} \C^5 |_{\mathfrak{spin}(8,\C)} \cong \bigwedge^{\mathrm{even}} \C^4 \oplus (\bigwedge^{\mathrm{odd}} \C^4 \wedge e_5),$$
which are representations of weight $\theta_1$ respectively $\theta_2$. By direct calculation we find
$$ V_{\omega_1} |_{\mathfrak{spin}(8,\C) \oplus \C \e_5} \cong V_{\theta_1} \otimes \C_{-1} \oplus V_{\theta_2} \otimes \C_{1}.$$
From \cite[Equation 25.34]{FulHar} and \cite[Equation 25.35]{FulHar} we obtain that $V_{\omega_1}|_{\mathfrak{spin}(6,\C) \oplus \C \e_5}$ is multiplicity-free, and the same holds for $V_{\omega_2}|_{\mathfrak{spin}(6,\C) \oplus \C \e_5}$. Hence the only possible representations are characters and the two spin representations.

Finally any representation
$\tau$
of $K = Spin(10) \times U(1) / \Gamma$ 
is a representations
 $\tau
=\pi\otimes \chi$ 
of $Spin(10) \times U(1)$ with $\Gamma
\subset \ker
(\pi\otimes \chi)$.
Writing $U(1)$
as $U(1) = \{ e^{t Z} \ | \ t \in [0,2 \pi] \}$ 
and the character $\chi=\chi_n: e^{tZ}
\to e^{itn},
$
the condition 
$\Gamma
\subset \ker
(\pi\otimes \chi)$
forces $n = 1 \mod 4$, or 
$e^{int} \pi(x)$
for $\pi=\pi_{\omega_1} $
and  $n = 3 \mod 4$
for $\pi=\pi_{\omega_2} $.
This finishes the proof.
\end{proof}

\subsection{$(G,K) = (E_{7(-25)}, E_6 U(1))$}
This is similar to the above, the group is chosen
to be the adjoint group
of the real Lie algebra
$\mathfrak e_{7(-25)}$
and its maximal compact
subgroup $K$ is realized
by its adjoint action on the real space $\mathfrak p$
and further as complex linear transformation 
on $\mathfrak p^+=\mathbb C^{27}$.
We claim that  $M=Spin(8)$.

The Lie algebra 
$\mathfrak k^{\C}
         =\mathfrak e_6^{\mathbb C} 
         + \mathbb C Z$ and
         the root system 
         of $
\mathfrak{e}_6^{\mathbb C}$  is given by
           $$\{\e_i\pm \e_j; 5\ge i>j\ge 
1\}\cup
\{\dfrac 12 (\sum_{i=1}^5 
(-1)^{\nu_i} 
\e_i -\e_6-\e_7 +\e_8); 
\text{$\sum_{i=1}^5 {\nu_i} $ is even}\},
$$
where we use the realization from \cite[Section 13]{EHW}. The real rank of $E_{7(-25)}$ is $3$.
   Its normalizer
   is the Lie algebra
   $\mathfrak{spin}(8, \mathbb C)$
   given by the root system
        $\{\e_i\pm \e_j; 4\ge i>j\ge 
1\}$
We directly calculate the Peirce decomposition \cite[Chapter 4]{loos}
   $$
   \mathfrak p^+ =V
   = V_{11} \oplus V_{12} 
   \oplus V_{13} 
   \oplus V_{22} 
\oplus   V_{23} 
   \oplus V_{33}
   =
 \mathbb Ce_1 \oplus 
   \mathbb C^8 \oplus
    \mathbb C^8 \oplus
  \mathbb Ce_2 \oplus     \mathbb C^8 
\oplus \mathbb Ce_3,
   $$
corresponding
   to one defining
   representation
   of $\mathfrak{so}(8, \mathbb C)
   =\mathfrak{spin} (8, \mathbb C)$
   given by the roots
   $$
\{\pm \e_i +\e_6; 1\le i\le 4 \}
$$
and two spin representations given
by the roots
$$ \{
\dfrac 12(
\sum_{i=1}^5 (-1)^{\nu(i)} \e_i
+\e_6-\e_7 +\e_8
); \text{$\sum_{i=1}^5 {\nu(i)}$ is odd}\}.
$$

The Lie algebra $\mathfrak m$
is the subalgebra of $\mathfrak k$
preserving this decomposition and it is
$\mathfrak m=
\mathfrak{spin} (8)
$
with $\mathfrak{spin} (8)$
acting on the 3 copies  $\mathbb C^8$. Moreover, the Jordan algebra is $ \mathfrak{p}^+ = H_3(\mathbb{O}_{\C})$, the Hermitian $3 \times 3$ matrices of bioctonions with respect to the canonical involution \cite[Chapter 4]{loos}. Tripotents are given by $e_i = E_{ii}$, and the $V_{ij}$ are given by $V_{ij} = \{ v E_{ij} + \tilde{v} E_{ji} \ | \ v \in \mathbb{O}_{\C} \}$. They are invariant under $M$ and
\begin{flalign*}
& (m (v) m (w) ) E_{13} + ( m (\tilde{w}) m( \tilde{v} ) ) E_{31}
\\ & = \frac{1}{4} \{m \cdot  (v E_{12} + \tilde{v} E_{21}), e_2, m \cdot (w E_{23} + \tilde{w} E_{32}) \}
\\ & = m \cdot  (v w E_{13} + \tilde{w} \tilde{v} E_{31}),
\end{flalign*}
in terms of triple products and octonion matrices.
Now we see that $\p = H_3(\mathbb{O})$ and there is an action of $M$ on three copies of the real division algebra $\mathbb{O}$, embedded in $V_{12}$, $V_{23}$ and $V_{13}$, such that
$$m(v) m (w) = m(vw).$$
Thus we get a map $M \rightarrow SO(8) \times SO(8) \times SO(8)$ which is injective and a triality. By \cite{desapio} the group of all trialities is equal to $Spin(8)$. Hence $M$ must be equal to $Spin(8)$. 

Now we prove the Theorem \ref{KresThm} (6).

\begin{proof}
As $K=E_6U(1)$ is connected and $M = Spin(8)$ it suffices to study the restriction of
the Lie algebra
representations
and their complexifications.
(More precisely
$K= E_6\times U(1)/\mathbb Z_3$
realized as subgroup
of $E_{7(-25)}$
acting on 
$\mathfrak p^+$
as complex linear transformations, where
$\mathbb Z_3$
is the subgroup 
$
\{(
e^{ik\frac
{2\pi}{3}},
e^{-ik\frac
{2\pi}{3}}),
k=0, 1, 2\}
$, but this fact is
not needed here.)
The subalgebra 
$\mathfrak{m} \subseteq \mathfrak{k}$ and be considered
as a subalgebra
of $\mathfrak{m}^{\mathbb C}
+\mathbb C^2 \subseteq \mathfrak{k}^{\mathbb C}
$ where 
$\mathbb R^2$ is in the Cartan subalgebra of 
$\mathfrak{k}^{\mathbb C}$.
The restriction 
of representations
of
$\mathfrak{k}^{\mathbb C}
$ to this equal rank subalgebra has
been studied in 
 \cite[Corollary 1.2.E6 (vi)]{St}, and 
it follows that the only possible options for which $\tau|_M$ could be multiplicity-free are representations of highest weight $n \omega_1$ or $n \omega_6$ of $E_6$ (we can ignore $SO(2)$); these are dual to each other by the symmetry of the Dynkin diagram of $E_6$, see e.g. \cite{Va}. Hence it suffices to study representations of highest weight $n \omega_1$. By a direct computation from \cite[page 513]{knapp} we see that $\p^{+}$ is the representation of highest weight $\omega_1$. Let the strongly orthogonal roots be $\{\gamma_1, \gamma_2, \gamma_3 \}$, where $\gamma_1$ is the highest and $\gamma_3$ the lowest root. By definition of $M$ as the commutator of $\af = \sum_{j=1}^3 \R(e_j + e_{-j})$ it acts trivially on $e_1, e_3 \in \p^+$. As $e_1$ and $e_3$ are connected to the highest and lowest weight of $\p^+$, both $e_1^{\otimes n}, e_3^{\otimes n} \in V_{n \omega_1} \subseteq V_{\omega_1}^{\otimes n}$. Now when we restrict a representation of highest weight $n \omega_1$ to $\mathfrak{m}$ we get the trivial representation at least twice.
\end{proof}

\section{
Induced representations, 
multiplicities, eigenfunctions
and eigenvalues
of invariant differential operators}
\label{sect-4}
We assume now
 the representations
$\tau$ of $K$
  are those
  classified in the previous section. We prove the uniqueness of  embeddings of finite-dimensional
  representations
  $W_{\lambda}
$ of $G$ in $C^{\infty}(G/K,V_{\tau})$ and their
 factorizations
 via the Poisson-Szeg\"o{} transform.

\subsection{Induced representations, the Poisson-Szeg\H{o} transform
and the realization of
finite-dimensional representations
in induced representatons
and in $C^\infty(G/K,V_{\tau})$}

We recall 
\cite{KW,Ya}
the Poisson-Szeg\H{o} transform 
$$
S_{\sigma, \nu}:
I_{\sigma, \nu}\to
C^\infty(G/K,V_{\tau}),
\, S_{\sigma, \nu}(f)(x) =\frac{\dim\tau}{\dim\sigma}\int_{K}\tau(k)f(xk)dk.
$$
It intertwines the
defining actions
of $G$ on the respective spaces.
We shall also need another
integral formula for $S_{\sigma, \nu}$ \cite{KW},
$$
S_{\sigma, \nu}f(x)
=
\int_K \Psi_{\tau, \nu}(k^{-1}x) f(k)dk
$$
where 
$\Psi_{\nu,\tau}: G \rightarrow \mbox{End }V_{\tau},$ is defined 
by
\begin{equation}
\label{Psi}
\Psi_{\nu,\tau}(x)=e^{(-\nu+\rho_{\mathfrak{a}})(H(x))}\tau(k(x))^{-1}, \ \forall \  x \in G    
\end{equation}
using  the Iwasawa decomposition $x\in G = NAK$,
$$
x=n(x)e^{H(x)}k(x).
$$

We recall also
from Equation (\ref{Jlambda})
the 
representations
of 
a finite-dimensional
$(W_\lambda,
G^{\C})$
in the induced
representations
$
J_\lambda: W_\lambda \to I_{\sigma_{\lambda},- \lambda|_{\mathfrak{a}} - \rho_{\mathfrak{a}}}^{0},
$
where $\sigma_{\lambda} = \sigma$ is the irreducible $M$-submodule of $\overline{N}$-fixed vectors in $W_{\lambda}$. We fix a $K$-invariant
inner product on $W_\lambda$.

\begin{proposition}
\label{finiterepprop}
(1) Let $(W_{\lambda}, G^{\C})$ be a highest weight representation of $G^{\C}$ and its Lie algebra $\mathfrak{g}^{\C}$. For the multiplicities we have $[V_{\tau} : W_{\lambda}|_K] \leq 1$, and if $V_{\tau} \subseteq W_{\lambda}|_K$ we have $[W_{\lambda} : I_{\sigma,-\lambda|_{\af} - \rho_{\af}}^0 ] \leq 1$, viewed as Lie algebra representations.  

(2)
Let $(W_\lambda, G ^{\C})$ be
a finite-dimensional
irreducible representation
such that 
$V_{\tau} \subseteq W_{\lambda}|_K$. Up to  constants there is a unique $G$-equivariant map
$$F : W_{\lambda} \rightarrow C^{\infty}(G/K,V_{\tau})$$
defined by
$$F(w)(g) = P_{\tau} (g^{-1} \cdot w), g\in G, w\in W,$$
where $P_{\tau}$ is the projection onto $V_{\tau}$. Moreover, if we normalize $J_{\lambda}$ so that for all $v \in V_{\tau}$ we have $J_{\lambda}(v) = f_v$, where $f_v(kan) =  a^{\lambda|_{\mathfrak{a}}} P_{\sigma} \tau(k)^{-1} v$,
then there is a factorization
of $F$ 
as a product $G$-equivariant maps,
$$
F=
S_{\sigma, -\lambda|_{\mathfrak{a}} - \rho_{\mathfrak{a}}}
J_{\lambda}: (W_\lambda, G^{\C})\to I_{\sigma, -\lambda|_{\mathfrak{a}} - \rho_{\mathfrak{a}}}^0
\to C^\infty(G/K,V_{\tau}).
$$
\end{proposition}

\begin{proof}
(1) We have
$
W_{\lambda} 
\subseteq I_{\sigma,-\lambda|_{\mathfrak{a}} - \rho_{\mathfrak{a}}}^0$ and then
$[V_{\tau}: 
W_{\lambda}]\le 
[V_{\tau} : 
I_{\sigma,-\lambda|_{\mathfrak{a}} - \rho_{\mathfrak{a}}}^0|_K ] = [V_{\tau} : Ind_M^K(\sigma)] = [V_{\tau}|_M : \sigma] \leq 1$, by our conditions on $\tau$. When $V_{\tau} \subseteq W_{\lambda}|_K$ the multiplicity $[W_{\lambda} : I_{\sigma,-\lambda|_{\mathfrak{a}} - \rho_{\mathfrak{a}}}^0 ] \leq 1$ as $I_{\sigma,-\lambda|_{\mathfrak{a}} - \rho_{\mathfrak{a}}}^0|_K$ contains $V_{\tau}$ only once.

(2)
The first part is
an elementary fact.
For any $G$-equivariant map $R:
W_\lambda \to 
C^\infty(G/K, V_\tau)
$,
$$
w
\in W_\lambda
\mapsto
R(w) \in 
C^\infty(G/K, V_\tau)
\mapsto R(w)(o) \in V_\tau, o=K\in G/K,
$$
is a $K$-equivariant
map and thus must 
be constant by (1); namely
$Rw$ is uniquely
determined by $Rw(o)$. This
proves that $R$
is given by $F$,
up to a constant.

Now both $F$
and $
S_{\sigma, -\lambda|_{\mathfrak{a}} - \rho_{\mathfrak{a}}}
J_{\lambda}$
are $G$-equivariant
maps, so
they are the same
up to a constant.
We evaluate the map $S_{\sigma, -\lambda|_{\mathfrak{a}} - \rho_{\mathfrak{a}}} J_\lambda(v)
$ at the identity $g=1$ fo $v \in V_{\tau}$,
\begin{equation*}
S_{\sigma, -\lambda|_{\mathfrak{a}} - \rho_{\mathfrak{a}}}J_\lambda(v)
(1)=
\frac{\dim\tau}
{\dim\sigma}
\int_K \tau(k) P_\sigma \tau(k^{-1}) v dk
   =  
   \frac{\dim\tau}
{\dim\sigma}   \frac{\dim\sigma}{\dim\tau} v
=F(v)(1)
\end{equation*}
by the Schur orthogonality relations.
Thus $S_{\sigma, -\lambda|_{\mathfrak{a}} - \rho_{\mathfrak{a}}}J_\lambda
v(g)=F(v)(g)$ for 
any $v \in W_{\lambda}$ and $g\in G$ as claimed. 

 \end{proof}

\begin{remark}
As mentioned in
 the Introduction,
  it follows from 
e.g.  \cite[Prop. 2.2, (40)]{Campo}
that the multiplicity
free result
holds also for 
discrete 
series 
representation
$(\pi, G)$
in 
the space
$L^2(G/K, V_\tau)$
of $L^2$-sections:
if $\pi$ appears 
discretely in
$ L^2(G/K, V_\tau)$,
equivalently
if $\pi$ is a discrete
series and
$[\tau: \pi|_K]\ne 0$, 
then
the multiplicity
$[\pi: L^2(G/K, V_\tau)] \leq 1$. It is therefore
an interesting question
to characterize 
induced representations and discrete series
which contain
the $K$-type $\tau$
classified in Section 3 above.
\end{remark}

 \subsection{Invariance
 of 
 the eigenvalues of differential operators
}

We want to prove invariance properties on the eigenvalues of differential operators. These invariance properties have been used in \cite{SZ} to compute the eigenvalues of Shimura differential operators for one-dimensional representations $\tau$. We first state some facts derived from \cite[Section 3]{L}. Note that from the Iwasawa decomposition for the Lie algebra we have the decomposition 
$$U(\mathfrak g^{\mathbb C}
)
=U(\mathfrak a^{\mathbb C}) U(\mathfrak k^{\mathbb C})\oplus
\mathfrak n^{\mathbb C}
U(\mathfrak g^{\mathbb C}).$$
Let 
$$
\Omega:U( \mathfrak g^{\mathbb C})\rightarrow U( \mathfrak a^{\mathbb C}) U(\mathfrak k^{\mathbb C})
$$
be the correponding
projection. We give $U(\mathfrak a^{\mathbb C}) U(\mathfrak k^{\mathbb C})$ an algebra structure by identifying it with the algebra $U(\mathfrak a^{\mathbb C}) \otimes U(\mathfrak k^{\mathbb C})$ and regard $\Omega$ as a map to $U( \mathfrak a^{\mathbb C}) \otimes U(\mathfrak k^{\mathbb C})$. Then $\Omega(uv)=\Omega(v)\Omega(u), $ for any $u \in U(\mathfrak g^{\mathbb C}), v \in (U(\mathfrak g^{\mathbb C})^{K}$  and $$\Omega(U(\mathfrak g^{\mathbb C})^{M})\subseteq  U( \mathfrak a^{\mathbb C}) \otimes U(\mathfrak k^{\mathbb C})^{M}.$$
Let $T:U(\mathfrak k^{\mathbb C})\rightarrow U(
\mathfrak k^{\mathbb C})$ be the canonical anti-automorphism of $U(\mathfrak k^{\mathbb C})$ defined by :
$$T(1)=1, T(x)=-x, T(xy)=T(y)T(x).$$
Put $$\Omega_{\tau}=(1\otimes \tau) \circ (1 \otimes T)\circ \Omega$$
Then $ \Omega_{\tau}:U(\mathfrak g^{\mathbb C})^{K}\rightarrow 
U(
\mathfrak a^{\mathbb C}
)\otimes \mbox{End}_{M}V_{\tau}$  is an algebra homomorphism.
Any $\lambda \in \mathfrak (a^{\mathbb C})'$
can be  extended to an evaluation on 
$
U(\mathfrak a^{\mathbb C})$, which we also  denote  by $\lambda$.
Let
$$\Omega_{\tau, \lambda}=(\lambda\otimes 1)\circ \Omega_{\tau}:U(\mathfrak g^{\mathbb C})^{K}\rightarrow  \mbox{End}_{M}V_{\tau}.$$

The following lemma follows
by direct
computations.

\begin{lemma}
\label{egvalue}
For $v \in V_{\tau}$ let $f_{v, \lambda}\in C^\infty(G/K, V_\tau)$
be defined by $f(x) = \Psi_{\tau, \lambda}(x) v, x\in G$,
where
$\Psi_{\tau, \lambda}
$ is defined in (\ref{Psi}). For any $D \in U(\mathfrak{g}^{\C})^{K}
\subset
(
U(\mathfrak{g}^{\C}\otimes End(V_\tau))^{K}
$
realized 
as  differential
operator on $C^\infty(G/K, V_\tau)$
via $\nabla(D)$ in
(\ref{diffopUnAlg})
we have $$(\nabla(D)f_{v, \lambda})(x)= \Psi_{\tau, \lambda}(x) \Omega_{
\tau, -\lambda+\rho_{\mathfrak{a}}
}(D)v, \ \forall \ x \in G.$$
 \end{lemma}

Recall the Weyl group $W=M'/M$. There is a natural action of $W$ on the dual $\widehat{M}$ 
of all irreducible representations of $M$ as follows. Let $\gamma\in \hat{M}$
and $s\in W=M'/M$ be represented by $\hat {s}\in M'$, then the representation $s. \gamma$ is
 $$
 (s. \gamma)(m)=\gamma (\hat {s}m \hat{s}^{-1}), \quad  m\in M.
$$
We state the main theorem. 

\begin{theorem}
\label{egvaluethm}
Assume the irreducible finite-dimensional  representation $V_\tau$  
of $K$ is  multiplicity-free  when  restricted to  $M$.

(1) For any $M$-irreducible subrepresentation $(\sigma,U_\sigma)$ in $\tau|_M$,  $f \in I_{\sigma,\nu}^0$ and $D \in \mathcal{D}^G(G/K, V_{\tau})$, we have $DS_{\sigma, \nu}(f) = \omega(D)(\sigma,\nu + \rho_{\mathfrak{a}})S_{\sigma, \nu}(f)$ for a constant $\omega(D)(\sigma,\nu + \rho_{\mathfrak{a}})$.
Furthermore, we have the invariance property:
$$\omega(D)(\sigma,\nu + \rho_{\mathfrak{a}}) = \omega(D)(s \cdot \sigma,s \cdot \nu + \rho_{\mathfrak{a}}), \ \forall \ s\in W.$$

(2) Let $(W_\lambda, G^{\C})$
be a  finite-dimensional
representation of $
G^{\C}$ containing $V_\tau$.  Let $D \in \mathcal{D}^G(G/K,V_{\tau})$, then $D$ acts on 
$F(W_\lambda)$
as the constant
$$
DF(w) = \omega(D)(\sigma_{\lambda},-\lambda|_{\mathfrak{a}})F(w), 
\quad  \ \forall \ w \in W_\lambda.
$$
\end{theorem}

\begin{proof}(1) Recall that any invariant differential operator $D \in \mathcal{D}^G(G/K, V_{\tau})$ can be represented by an element of $U(\mathfrak{g}^{\mathbb C})^K$ via
the natural
map (\ref{diffopUnAlg});
see also \cite[Proposition 2.1]{Sh},
and 
$\Omega_{\nu,\tau}(D)$ 
is well-defined using
the representative in $U(\mathfrak{g}^{\C})$.
As $\tau|_M$ is multiplicity-free and $\Omega_{-\nu + \rho_{\mathfrak{a}}}(D)|_{U_{\sigma}} \in \mbox{End}_{M}(U_\sigma) = \C I_{U_{\sigma}}$,  by Lemma \ref{egvalue},
we have $D.S_{\sigma, \nu}(f) =\Omega_{-\nu + \rho_{\mathfrak{a}},U_{\sigma}}(D)S_{\sigma, \nu}(f) = \omega(D)(\sigma, \nu + \rho_{\mathfrak{a}}) S_{\sigma,\nu}(f)$.
The invariance property follows directly from
\cite[Theorem 9.8(2)]{L}.

(2) The   follows from (1), since $F = S_{\sigma,-\lambda|_{\mathfrak{a}} - \rho_{\mathfrak{a}}} J_{\lambda}$ by Proposition 4.1. 
\end{proof}


\end{document}